\documentclass[11pt]{article}

\usepackage[utf8]{inputenc}
\usepackage{graphicx}
\usepackage{amsmath}
\usepackage{amssymb}
\usepackage{amsfonts}
\usepackage{amsthm}
\usepackage[margin=1in]{geometry}
\usepackage[font=small,width=\linewidth]{caption}
\usepackage[hidelinks]{hyperref}
\usepackage[title]{appendix}

\usepackage[linesnumbered, algoruled, vlined, algo2e]{algorithm2e}
\SetKwInput{KwInput}{Input}
\SetKwInput{KwOutput}{Output}
\usepackage{stmaryrd}%

\numberwithin{equation}{section}

\theoremstyle{plain}

\newtheorem{theorem}{Theorem}[section]
\newtheorem{lemma}[theorem]{Lemma}
\newtheorem{corollary}[theorem]{Corollary}
\newtheorem{remark}[theorem]{Remark}

\newtheorem{proposition}[theorem]{Proposition}

\usepackage{mathabx}
\usepackage{enumitem}
\usepackage{booktabs}
\newcommand{\ra}[1]{\renewcommand{\arraystretch}{#1}}

\usepackage[
    backend=biber,
    natbib=true,
    style=numeric-comp,
    sorting=none,
    maxnames=2,
    giveninits=true,
    url=false,
    arxiv=pdf,
    doi=false,
    eprint=true,
    isbn=false,
    date=terse,
    block=ragged,
    maxbibnames=99
]{biblatex}
\DeclareMathOperator*{\argmin}{arg\,min} %
\DeclareMathOperator{\grad}{\nabla}
\DeclareMathOperator{\ddiv}{div}
\DeclareMathOperator{\sep}{:}  %

\newcommand{\discretized}[1]{\ensuremath{\boldsymbol{#1}}}

\newenvironment{keywords}{%
    \begingroup%
    \small%
    \textbf{Key words.}}{\endgroup}
\newenvironment{AMS}{%
    \begingroup%
    \small%
    \textbf{AMS subject classifications.}}{\endgroup}

\title{Adaptive non-intrusive reconstruction of solutions to high-dimensional parametric PDEs}
\author{Martin Eigel, Nando Farchmin, Sebastian Heidenreich and Philipp Trunschke}
\date{\today}

\begin{document}

\maketitle

\begin{abstract}
    Numerical methods for random parametric PDEs can greatly benefit from adaptive refinement schemes, in particular when functional approximations are computed as in stochastic Galerkin and stochastic collocations methods.
    This work is concerned with a non-intrusive generalization of the adaptive Galerkin FEM with residual based error estimation.
    It combines the non-intrusive character of a randomized least-squares method with the a posteriori error analysis of stochastic Galerkin methods.
    The proposed approach uses the Variational Monte Carlo method to obtain a quasi-optimal low-rank approximation of the Galerkin projection in a highly efficient hierarchical tensor format.
    We derive an adaptive refinement algorithm which is steered by a reliable error estimator.
    Opposite to stochastic Galerkin methods, the approach is easily applicable to a wide range of problems, enabling a fully automated adjustment of all discretization parameters.
    Benchmark examples with affine and (unbounded) lognormal coefficient fields illustrate the performance of the non-intrusive adaptive algorithm, showing the expected convergence rates of single-level strategies.
\end{abstract}



\section{Introduction}%
\label{sec:intro}

High-dimensional parametric \emph{partial differential equations} (PDEs) play a crucial role in modern simulation methods used in the natural sciences and engineering.
Especially when uncertainties or variations in the data should be incorporated into the physical model, a parameter vector determining the data realizations leads to a discretization complexity that easily becomes extremely challenging to tackle due to the inherent ``curse of dimensionality''.
There has been very active research activity in particular in the area of Uncertainty Quantification (UQ) to better understand the structure of the problem and to mitigate the numerical obstacles with new methods.
Apart from sampling methods for the estimation of quantities of interest such as Monte Carlo sampling, functional approximations allow to exploit commonly encountered structured regularity of the PDE solutions to obtain much higher convergence rates.
The central contribution of this paper is the development of an adaptive sample-based Galerkin method in low-rank tensor format, which can be considered a non-intrusive generalization of the adaptive stochastic Galerkin FEM (ASGFEM) as e.g. derived in~\cite{EigelGittelson2014asgfem,EigelMarschall2018,EigelMarschall2020lognormal}.
It combines and generalizes previous results on the residual based reliable error estimator with the non-intrusive low-rank tensor reconstruction techniques from~\cite{EigelTrunschke2019vmc,EigelTrunschke2020} and~\cite{EigelFarchmin2021expTT}.
In contrast to the frequently used intrusive ASGFEM, the presented method is a versatile generalization that could easily be applied to a broad range of problems with only small modifications.
A main feature it shares with stochastic collocation (SC) methods is its sole dependence on pointwise solutions.
However, opposite to SC, a hierarchical tensor compression of the solution and coefficient field usually lead to a beneficial scaling with respect to the parameter dimensions, allowing to compute very high-dimensional problems that might otherwise only be tractable by neural network representations.

As a model problem we consider the parameter dependent Darcy equation
\begin{equation}
    \begin{aligned}
      \label{eq:setup:darcy}
      -\ddiv_x a(x,y) \grad_x u(x,y) &= f(x) \quad\mbox{in }D,\\
      u(x,y) &~= 0 \qquad\mbox{on }\partial D,
    \end{aligned}
\end{equation}
on some domain $D\subset\mathbb{R}^2$, where $y=(y_1,\dots,y_L)\subset\mathbb R^L$ is a high or even infinite dimensional parameter vector determining the (affine or nonlinear in $y$) coefficient field and hence the solution.
To counter the resulting (possibly) extensive memory complexity caused by a functional representation, the coefficient tensors of $a(x,y)$ and $u(x,y)$ are low-rank approximated in the \emph{Tensor Train} (TT) format via a least-squares regression technique coined the \emph{Variational Monte Carlo} (VMC) method~\cite{EigelTrunschke2019vmc}.
Given a sufficient amount of training samples, which can be estimated by a heuristic criterion, the obtained approximation is equivalent to the stochastic Galerkin projection with high probability~\cite{EigelTrunschke2020,Trunschke2021}.
In comparison to Monte Carlo methods, exploiting the regularity and low-rank approximability of the solution may drastically improve the rate of convergence as shown in~\cite{EigelTrunschke2019vmc}.

The resulting adaptive algorithm only requires access to samples $u(y^{(i)})$ (and possibly $a(y^{(i)})$) generated via a black-box solver.
Given an approximate solution reconstruction $w_N\approx u$, contributions of the deterministic, stochastic and algebraic errors can be evaluated, leading to a reliable upper bound of the energy error of the form
\begin{align}
  \label{eq:setup:est_total}
  \Vert u-w_N\Vert \lesssim \eta_\mathrm{det}(w_N) + \eta_\mathrm{sto}(w_N) + \eta_\mathrm{alg}(w_N).
\end{align}
These error estimates then steer the local refinements of the discrete space.
It is noteworthy that no direct interaction with the parametric solver is required besides communicating how the underlying mesh has to be changed.
The error estimators can be proven to be reliable and efficient in many cases, see~\cite{EigelGittelson2014asgfem,EigelGittelson2014convergence,BespalovPraetorius2019convergence,BespalovPraetorius2021optimality} and~\cite{Carstensen2012review} for a review of the underlying deterministic derivation.

We demonstrate the performance of the proposed algorithm for the model problem~\eqref{eq:setup:darcy} with both affine and lognormal diffusion coefficients $a(x,y)$.
It should be pointed out that our algorithm can be applied with only minor modifications to any linear or (mild) nonlinear problem that permits computable a posteriori error bounds.

\paragraph{Structure}

Section~\ref{sec:setup} introduces the model problem setting, its variational formulation as well as the spatial and stochastic discretization.
Section~\ref{sec:discretization} then examines how the diffusion field and the solution of~\eqref{eq:setup:darcy} can be obtained efficiently and non-intrusively.
In Section~\ref{sec:estimator}, we recall the residual based error estimator from~\cite{EigelGittelson2014asgfem,EigelMarschall2020lognormal} and derive a heuristic for controlling the regression error.
The resulting adaptive refinement strategy and the overall algorithm are presented in Section~\ref{sec:algorithm}.
Finally, we test the fully adaptive scheme in several examples with affine and lognormal coefficient fields in Section~\ref{sec:experiments}.

\paragraph{Related work}

Theoretical considerations about the class of parametric PDEs used in this paper~\eqref{eq:setup:darcy} can e.g.\ be found in the review articles~\cite{SchwabGittelson2011,Cohen2015}.
Moreover, the technically involved lognormal case is analysed in detail in~\cite{GalvisSarkis2009,Mugler2013phd,Bachmayr2017sparse}.

Functional representations of high-dimensional problem solutions can be obtained by spectral approximations, enabling optimal convergence rates numerically.
Most prominent are SC~\cite{Nobile2008,Nobile2008sparse,Ernst2014} and SG~\cite{LeMaitreKnio2010,GhanemSpanos1990} methods.
SC has the advantage of being non-intrusive and thus is easy to use with already existing simulation codes, whereas SG methods can be understood as an extension of classical finite element (FE) methods, requiring a problem-specific implementation.
To make these methods computationally feasible, different model reduction techniques can be incorporated.
In the context of this paper, two are most relevant: (i) adaptivity based on computable error estimators and (ii) low-rank compression with hierarchical tensor formats.

Adaptive algorithms based on a posteriori estimators have been developed to increase the physical FE space and the stochastic space automatically and problem dependent.
They can be seen as an extension of adaptive methods in deterministic FEM, see~\cite{EigelGittelson2014asgfem,EigelGittelson2014convergence,EigelMerdon2016,Tempone2022} for residual based estimators and~\cite{Bespalov2014,Bespalov2016,Bespalov2018,Crowder2019} for hierarchical estimators.
Alternative adjoint approaches for quantities of interest can be found in~\cite{Bryant2015,PrudhommeBryant2015}.

Refinement of the stochastic space typically incorporates enlarging the global polynomial basis by increasing the polynomial degrees and including more stochastic modes.
When using tensor formats, an additional error contribution that has to be controlled is the algebraic (compression) error, leading to an adjustment of the representation rank.
In case of affine coefficient fields, there are first results on the convergence of adaptive algorithms~\cite{EigelGittelson2014convergence,BespalovPraetorius2019convergence} and even optimality~\cite{BespalovPraetorius2021optimality} under certain conditions.
Recent results provide optimality in much greater generality by using wavelet expansions~\cite{Bachmayr2021adaptive}.

To circumvent exponential growth of the stochastic discretization space, hierarchical tensor formats can be used if the problem is low-rank representable.
Some details on different tensor formats and numerical algorithms can be found in~\cite{Bachmayr2016,Hackbusch2014,Nouy2017,Oseledets2009}. 
Our focus lies on the TT format~\cite{Oseledets2009}, which has been used with tremendous success for the solution of parametric PDEs and related UQ problems such as Bayesian inversion and random field representations, see~\cite{Dolgov2015,DolgovScheichl2019alsCross,Dolgov2019sampling,EigelGruhlke2020bayes,Dolgov2014}.
ASGFEM in hierachical tensor formats are presented in~\cite{EigelMarschall2018,EigelMarschall2019randomDomains} for affine coefficients.
The first ASGFEM for lognormal coefficients is developed in~\cite{EigelMarschall2020lognormal}.
In contrast to affine fields, nonlinear expansions as in the lognormal case cannot be represented easily in tensor formats and are in fact rather challenging to obtain.
Accordingly, a limitation of the ASGFEM in~\cite{EigelMarschall2020lognormal} is that it relies on a specifically tailored construction of the parametric field.
Alternatively, methods based on tensor completion~\cite{DolgovScheichl2019alsCross, Espig2014} or Galerkin projection~\cite{EigelFarchmin2021expTT} have been developed in recent years.
The approach presented in~\cite{EigelFarchmin2021expTT} even allows to govern the approximation error with quantities computed during the approximation, avoiding any computational overhead.
Since this can be employed with a wide range of coefficient fields, it is very much in the spirit of our method and is used in Section~\ref{sec:discretization}.
The basis for the least-squares tensor regression of parametric PDE solutions was laid with the VMC method in~\cite{EigelTrunschke2019vmc}.
An alternative technique is the tensor cross approximation of~\cite{Oseledets2010}. 
From a practical point of view, a major difference of the two methods lies in the integration of training samples.
The cross approximation evaluates the parametric black-box solver during runtime as the interpolation points are chosen adaptively (``active learning''), whereas the VMC method assumes precomputed (randomly sampled) evaluations of the solver for random parameter realizations (``passive learning'').

\section{The model problem}%
\label{sec:setup}

This section establishes the analytical foundations of the model problem~\eqref{eq:setup:darcy}.
We recall some details on the functional setting for two common types of random coefficient fields and point to references for an in depth analysis when necessary.
Throughout this work we assume $D\subset \mathbb{R}^2$ to be a polygonal bounded Lipschitz domain.
Moreover, without loss of generality, we limit ourselves to a deterministic source term $f\in L^2(D)$ and homogeneous Dirichlet boundary conditions since modelling the right-hand side and the boundary conditions as stochastic fields independent on the diffusion coefficient $a(x,y)$ would not introduce significant modifications.
With typical applications in e.g.\ stochastic groundwater flow modelling, the diffusion coefficient is often defined by a Karhunen-Lo\`eve type expansion of the form
\begin{align}
  \label{eq:setup:gamma}
  \gamma(x,y) = \gamma_0(x) + \sum_{\ell=1}^{L} \gamma_\ell(x) y_\ell
  \qquad\mbox{for } x\in D,
\end{align}
and almost all $y_\ell\sim\pi_\ell$ with independent distributions $\pi_\ell$, $L\in\mathbb{N}\cup\{ \infty \}$.
In many applications, however, the far more challenging exponential diffusion field
\begin{align}
  \label{eq:setup:kappa}
  \kappa(x,y) = \exp( \gamma(x,y) - \gamma_0(x))
\end{align}
has to be considered.
Solvability of~\eqref{eq:setup:darcy} for $a=\gamma$ follows directly from the uniform boundedness and positivity of the affine field~\eqref{eq:setup:gamma}~\cite{SchwabGittelson2011}.
Well-posedness of~\eqref{eq:setup:darcy} for $a=\kappa$ with unbounded parameters $y$ is significantly more involved and requires the introduction of adapted function spaces, cf.~\cite{Bachmayr2017sparse,Hoang2014,Gittelson2010,Mugler2013,GalvisSarkis2009}.
We refer to~\cite{EigelMarschall2020lognormal} for a concise review of the concepts that we use for the problem setting~\eqref{eq:setup:darcy}.

Let $\mathcal{X}:=H_0^1(D)$ be equipped with the standard norm $\Vert w\Vert_\mathcal{X}=\Vert \grad w\Vert_{L^2(D)}$ and let $\mathcal{F} := \{ \mu\in\mathbb{N}_0^{\infty}\colon \vert \operatorname{supp}\mu\vert < \infty \}$ be the set of finitely supported multi-indices, where $\operatorname{supp}\mu$ denotes the set of all indices of $\mu$ different from zero. 
For any $m\in\mathbb{N}_0$ and $n\in\mathbb{N}$, let $[m\sep n]:=\{ m,\dots,n-1 \}$, where $[m\sep n] := \{ 0 \}$ if $m\geq n$ and $[n] := [0\sep n-1]$.
Define the full tensor index set
\begin{align}
  \label{eq:setup:tensor_index_set}
  \Lambda_d := [d_1] \times \dots \times [d_L] \times [1] \times \dots \subset \mathcal{F},
  \qquad\mbox{for }L\in\mathbb{N}\mbox{ and }d\in\mathbb{N}^L.
\end{align}
We refer to $d_\ell\geq1$ as the (stochastic) dimension for the mode $\ell\in L$ and call a mode $\ell$ active if $d_\ell>1$.
By $\{P^\ell_j\}_{j=0}^\infty$ we denote a set of orthogonal and normalized polynomials in $L^2(\Gamma_\ell,\pi_\ell)$, where we assume $\Gamma_{\ell}\subset\mathbb{R}$ for $\ell\in\mathbb{N}$.
Moreover, we consider the tensor product case $\Gamma=\prod_{\ell\in\mathbb{N}}\Gamma_\ell$ and $\pi=\prod_{\ell\in\mathbb{N}}\pi_\ell$ and define an orthonormal product basis $\{ P_\mu \}_{\mu\in\mathcal{F}}$ of $L^2(\Gamma,\pi)$ by $P_\mu(y) := \prod_{\ell\in\mathbb{N}} P_{\mu_\ell}^\ell(y_\ell) = \prod_{\ell\in\operatorname{supp}(\mu)} P_{\mu_\ell}^\ell(y_\ell)$.
Note that the use of global polynomials is justified by the high regularity of the solution of~\eqref{eq:setup:darcy} with respect to the random variables~\cite{Cohen2010,Cohen2011,Hoang2014}.
In our numerical experiments we rely on (scaled) Hermite and Legendre polynomials as univariate basis functions.
Details on the normalization constants for the respective polynomials and an analytical expression for the triple products $\tau_{ijk}=\mathbb{E}[P_i^{\ell}P_j^{\ell}P_k^{\ell}]$ are given in Supplement~\ref{supplement:normalization_and_triple_products}.
We define the bilinear form
\begin{align}
  \label{eq:setup:B}
  B(w,v) := \int_{\Gamma} \int_{D} a(x,y)\grad w(x,y)\cdot\grad v(x,y) \,\mathrm{d}x \,\mathrm{d}\pi(y)
\end{align}
on $L^2(\Gamma,\pi;\mathcal{X})$ and denote the induced energy norm by $\Vert w\Vert_{B} := B(w,w)^{1/2}$.
We additionally abbreviate $\Vert w\Vert_{\pi,D}=\Vert w\Vert_{L^2(\Gamma,\pi;L^2(D))}$.
The variational form of~\eqref{eq:setup:darcy} then reads
\begin{align}
  \label{eq:setup:varriational_form}
  B(u,v) = F(v) \quad\mbox{for all }v\in\mathcal{V},
\end{align}
where $F(v)=\int_{\Gamma}\int_D f(x)v(x,y)\,\mathrm{d}x\,\mathrm{d}\pi(y)$ is supposed to be well defined for an appropriate Hilbert space $\mathcal{V}$.

A conforming FE space $\mathcal{X}_p(\mathcal{T}):=\operatorname{span}\{ \varphi_i \}_{j=1}^N$ is used as spatial discretization of the physical space $\mathcal{X}$.
In particular, we assume $\mathcal{T}$ to be a regular triangulation of the domain $D$ with edges $\mathcal{E}$ and consider for $p\in\mathbb{N}$ the standard conforming order-$p$ Lagrange elements such that $\mathcal{X}_p(\mathcal{T}) := P_p(\mathcal{T})\cap C(\bar{D})$, where $P_p(\mathcal{T})$ is the space of element-wise polynomials of order $p$.
For any element $T\in\mathcal{T}$ and edge $E\in\mathcal{E}$, let $h_T$ and $h_E$ denote the diameter of $T$ and $E$, respectively.
Define the normal jump of a function $w\in H^1(D;\mathbb{R}^2)$ over the edge $E=\bar{T_1}\cap \bar{T_2}$ by $\llbracket w \rrbracket_E = (w|_{T_1} - w|_{T_2})\cdot \nu_E$ for the edge normal vector $\nu_E=\nu_{T_1}=-\nu_{T_2}$ of $E$.
Since the direction of the normal $\nu_E$ depends on the enumeration of the neighbouring triangles, we assume an arbitrary but fixed choice of the sign of $\nu_E$ for each $E\in\mathcal{E}$.
This allows us to define the fully discrete approximation space by
\begin{align}
  \label{eq:setup:fully_discrete_space}
  \mathcal{V}_N := \mathcal{V}_N(\Lambda_d; \mathcal{T}, p)
  := \Bigl\{ v_N = \sum_{\mu\in\Lambda_d} v_{N,\mu} P_\mu \mbox{ with } v_{N,\mu}\in\mathcal{X}_p(\mathcal{T}) \mbox{ for all } \mu\in\Lambda_d\Bigr\}
  \subset \mathcal{V}.
\end{align}
As a consequence, the Galerkin projection $u_N\in\mathcal{V}_N$ of the solution $u$ of~\eqref{eq:setup:varriational_form} is determined uniquely by
\begin{align}
  \label{eq:setup:galerkin_system}
  B(u_N,v_N) = F(v_N)
\qquad\mbox{for all }v_N\in\mathcal{V}_N.
\end{align}
To describe the lognormal case $a=\kappa$, define the set of admissible parameters
\begin{align}
  \label{eq:setup:Gamma_kappa}
  \Gamma_\kappa := \{ y\in\mathbb{R}^{L}\colon \sum_{\ell=1}^{L} \Vert \gamma_\ell\Vert_{L^{\infty}(D)} \vert y_\ell \vert < \infty \},
\end{align}
which is necessary and sufficient to guarantee pointwise boundedness and positivity of $\kappa$~\cite{Hoang2014}.
Note that for $\sigma_\ell>0$ the probability density function for the univariate Gaussian distribution $\mathcal{N}(0,\sigma_\ell^2)$ can be written as 
\begin{align*}
  \pi_{\ell}(y_\ell;\sigma_\ell) = \zeta_{\ell}(y_\ell;\sigma_\ell) \frac{1}{\sqrt{2\pi}}\exp(-\frac{1}{2}y_\ell^2)
  \quad\mbox{with}\quad
  \zeta_{\ell}(y_{\ell};\sigma_\ell)
  := \frac{1}{\sigma_\ell} \exp\Bigl( \bigl(\frac{1}{2} -\frac{1}{2\sigma_\ell^2}\bigr) y_\ell^2 \Bigr).
\end{align*}
Moreover, for any $\rho\geq0$ let $\sigma_\ell(\rho) := \exp(\rho\Vert\gamma_\ell\Vert_{L^\infty(D)})$ and let 
\begin{align}
  \label{eq:setup:zeta_and_pi}
  \zeta_{\rho}(y) := \prod_{\ell=1}^{L} \zeta_{\ell}(y_{\ell};\sigma_\ell(\rho))
  \qquad\mbox{and}\qquad
  \pi_{\rho} := \prod_{\ell=1}^{L} \pi_{\ell}(y_\ell;\sigma_\ell(\rho)).
\end{align}
Note that $\rho=0$ implies $\zeta_{0}=1$ and thus $\pi_0$ denotes the standard Gaussian density.
We henceforth investigate the following two application cases.
\begin{enumerate}[label=(\Alph*)]
  \item\label{case:setup:affine}
    \textbf{affine case.}
    Consider the affine coefficient field $a(\bullet,y)=\gamma(\bullet,y)\in L^\infty(D)$ $\pi$-a.e.
    In this setting, we assume the univariate random variables $y_\ell$ to be i.i.d.\ uniformly distributed on $\Gamma_{\ell}=[-1,1]$, i.e.\ $\pi_\ell=1/2$.
    Consequently, we set $\mathcal{V}=L^2(\Gamma,\pi;\mathcal{X})$ and employ Legendre polynomials as basis functions.
    For notational convenience, we additionally set $\zeta_\rho(y)\equiv1$ and $\pi_\rho\equiv\pi=2^{-L}$ for any $\rho\in\mathbb{R}$ for the analysis later on.
  \item\label{case:setup:lognormal}
    \textbf{lognormal case.}
    Consider the exponential coefficient field $a(\bullet,y)=\kappa(\bullet,y)\in L^\infty(D)$ $\pi$-a.e.
    Since we assume the univariate random variables $y_\ell$ to follow an i.i.d.\ standard normal distribution, we refer to $\kappa$ as a lognormal coefficient field.
    For $\vartheta\in[0,1]$, $\rho>0$ and $\Gamma=\Gamma_\kappa$, the solution space is defined by
    \begin{align*}
      \mathcal{V} = \{ w\colon\Gamma\to\mathcal{X} \mbox{ measurable with }B(w,w)<\infty \},
    \end{align*}
    where $B$ is the bilinear form~\eqref{eq:setup:B} with $\pi=\pi_{\vartheta\rho}$ from~\eqref{eq:setup:zeta_and_pi}.
    For the polynomial basis we choose scaled Hermite polynomials $\{ H_\mu^{\vartheta\rho} \}_{\mu\in\mathcal{F}}$, see Supplement~\ref{supplement:normalization_and_triple_products} and~\cite{SchwabGittelson2011}.
\end{enumerate}

\section{Discretization of solution and coefficient}%
\label{sec:discretization}

This section describes the approximation of the discrete solution $u_N$ of~\eqref{eq:setup:galerkin_system}, the right-hand side $f$ and the coefficient fields $\gamma$ and $\kappa$ in the TT format.
A brief summary of notation and some fundamental properties of the TT format are given in Supplement~\ref{supplement:tt_format}.
A general and more detailed description of the TT format is given in~\cite{Oseledets2009,Oseledets2011,Holtz2012} and in e.g.~\cite{EigelPfeffer2016,Bachmayr2016,DolgovScheichl2019alsCross,EigelMarschall2020lognormal} TT representations have been applied to the elliptic model problem~\eqref{eq:setup:darcy}.

\subsection{TT approximation of the solution}%

In the following we recall the notion of nonlinear least-squares approximation and show that a sample-based quasi-best approximation of $u$ can be obtained with high probability given sufficiently many samples.
For this, recall that the discrete solution $u_N$ of~\eqref{eq:setup:galerkin_system} satisfies the Galerkin orthogonality property $B(u_N-u, v_N) = 0$ for all $v_N\in\mathcal{V}_N$, which implies that $\Vert u-v_N \Vert_{B}^2 = \Vert u-u_N \Vert_{B}^2 + \Vert u_N-v_N \Vert_{B}^2$ for any $v_N\in\mathcal{V}_N$.
This means that $u_N$ is the $\Vert\bullet\Vert_{B}$-best approximation to the solution $u$ in $\mathcal{V}_N$.
Extending this idea to the subset $\mathcal{M}_r = \{w_N\in\mathcal{V}_N \colon \operatorname{tt-rank}(\discretized{w}) \le r\}$
leads to the best approximation problem
\begin{equation*}
    \argmin_{w_N\in\mathcal{M}_r} \Vert u-w_N \Vert_{B}^2
    = \argmin_{\substack{w_N\in\mathcal{M}_r}} \Vert u_N-w_N \Vert_{B}^2 .
\end{equation*}
Minimizing the energy norm is straight-forward for the affine case~\ref{case:setup:affine}, since $\mathcal{V}= L^2(\Gamma,\pi;\mathcal{X})$.
For the lognormal case~\ref{case:setup:lognormal} however, minimizing the $\Vert\bullet\Vert_{B}$-norm introduces an additional dependence on the diffusion coefficient.
Even though this is not problematic from a theoretical point of few, it is possible to eliminate this dependence by employing the boundedness of the bilinear form $B$~\eqref{eq:supplement:boundedness_B}, i.e.
\begin{align*}
  \Vert u_N - w_N \Vert_{B} 
  \leq \sqrt{\hat{c}(\vartheta\rho)} \Vert u_N - w_N \Vert_{L^2(\Gamma,\pi_{\rho};\mathcal{X})}.
\end{align*}
We then aim to find the minimum
\begin{equation}
    \label{eq:discretization:min}
    u_{N,r}
    := \argmin_{w_N\in\mathcal{M}_r} \Vert u_N-w_N \Vert_{L^2(\Gamma,\pi_{\rho};\mathcal{X})}^2,
\end{equation}
which has the advantage of being independent of the choice of $\vartheta$.
Since computing the $L^2$-norm with respect to $\pi_\rho$ is infeasible in practice, we follow the ideas of~\cite{EigelNeumann2019,EigelTrunschke2019vmc} and replace the high-dimensional integral with the Monte Carlo estimate
\begin{equation}
    \label{eq:discretization:empirical_norm}
    \Vert v\Vert_n := \sqrt{\frac{1}{n}\sum_{i=1}^n \Vert v(y^{(i)})\Vert_{\mathcal{X}}^2}
\end{equation}
for any $v\in\mathcal{V}$, where the samples $y^{(i)}\sim\pi_\rho$ are independent for all $i=1,\ldots,n$.
The computation of the best approximation then reads
\begin{equation}
    \label{eq:discretization:empirical_min}
    u_{N,r,n} := \argmin_{w_N\in\mathcal{M}_r} \Vert u_N-w_N \Vert_n^2,
\end{equation}
where the $u_N(y^{(i)})$ can be computed with an arbitrary FE solver.
The resulting \emph{nonlinear least squares} problem~\eqref{eq:discretization:empirical_min} is easy to implement and many highly optimized frameworks exist to solve this problem~\cite{EigelNeumann2019}.
Note that this is a fully non-intrusive approach since it only requires the (pointwise) standard FE solutions $u_N(y^{(i)})$ of the deterministic problem~\eqref{eq:setup:darcy}.
Theorem~2.12 in \cite{EigelTrunschke2020} ensures that the minimizer $u_{N,r,n}$ of~\eqref{eq:discretization:empirical_min} is comparable to the best-approximation $u_{N,r}$ of~\eqref{eq:discretization:min} given a sufficiently large number of samples $n$.
Moreover, in~\cite{EigelTrunschke2020} the authors derive a qualitative lower bound for the required number of samples $n$, which guarantees that $u_{N,r,n}$ is a quasi-best approximation on $\mathcal{M}_r$ with high probability.

\begin{remark}
  Strictly speaking~\cite[Theorem 2.12]{EigelTrunschke2020} holds only for the affine case~\ref{case:setup:affine} if we assume $y^{(i)}\sim\pi_\rho$.
  For the lognormal case~\ref{case:setup:lognormal}, an adapted sampling density (cf.~\cite{CohenMigliorati2017,EigelTrunschke2020}) is required.
  Nevertheless, in practice we do not observe that this is necessary.
  Moreover, we note that the sampling bound for $n$ established in~\cite{EigelTrunschke2020} is a worst-case bound and that a significantly smaller number of samples suffices in our experiments.
\end{remark}

\subsection{TT representation of the diffusion coefficient}%
\label{sec:discretization:diff_coeff}

\subsubsection*{Affine coefficient field and right-hand side}

The constant right-hand side $f$ and the affine coefficient field $\gamma$ are both given by an expansion into polynomials with at most one active mode each, i.e., each of the expansion terms is a univariate polynomial.
Such functions possess a natural (exact) representation of their coefficient tensor in the TT format.
For $d\in\mathbb{N}^L$, consider the set of univariate $L$-dimensional indices
\begin{align*}
  \Delta_d := \bigcup_{\ell=1}^L \{ j\, e_\ell \colon j\in[d_\ell] \}
  \qquad\mbox{with}\qquad \vert\Delta_d\vert = 1 + \sum_{\ell=1}^{L} (d_\ell-1),
\end{align*}
where $e_\ell$ denotes the unit vector $(e_\ell)_j = \delta_{\ell j}$ for $j\in\mathbb{N}$, and let $\iota\colon \{ 0,\dots,\vert\Delta_d\vert-1 \}\to\Delta_d$ with $\iota(0)=(0,\dots,0)$ be an arbitrary enumeration of $\Delta_d$.
Furthermore, let $\iota_{\mu_\ell}=\iota^{-1}(\mu_\ell e_\ell)$ be the enumeration index of $\mu_\ell e_\ell$ and define for $k\in\mathbb{N}_0$
\begin{align*}
  \delta(k,\mu_\ell) := \delta_{k 0}\delta_{\mu_\ell 0} + (1-\delta_{k 0})\bigl( \delta_{\mu_\ell 0} + \delta_{k\iota_{\mu_\ell}} (1-\delta_{\mu_\ell 0}) \bigr).
\end{align*}
Any function $w_N\in\mathcal{V}_N(\Delta_d;\mathcal{T},p)$ can formally be written as an expansion with respect to the full tensor set $\Lambda_d$,
\begin{align}
  \label{eq:discretization:univariate_function}
  w_N(x,y)
  = \sum_{t=0}^{\vert\Delta_d\vert-1} w_{\iota(t)}(x) P_{\iota(t)}(y)
  = \sum_{j\in[\vert\mathcal{T}\vert]} \sum_{\mu\in\Lambda_d} \discretized{w}[j,\mu] \varphi_j(x) P_\mu(y),
\end{align}
where $\discretized{w}[j,\mu]=0$ for $\mu\not\in\Delta_d$.
The following proposition shows that the coefficient tensor $\discretized{w}$ of~\eqref{eq:discretization:univariate_function} has an exact representation in the TT format.

\begin{proposition}%
  \label{pro:discretization:exactTT_gamma}
  For any $w_N\in\mathcal{V}_N(\Delta_d;\mathcal{T},p)$ the coefficient tensor $\discretized{w}\in\mathbb{R}^{N\times d}$ has an exact representation in the TT format.
  This representation is given by
  \begin{align*}
    \discretized{w}[j,\mu]
    = \sum_{k_1=0}^{\vert\Delta_d\vert-1} \dots \sum_{k_M=0}^{\vert\Delta_d\vert-1} \discretized{w}_0[j,k_1]\prod_{m=1}^M \discretized{w}_m[k_m,\mu_m,k_{m+1}],
  \end{align*}
  with spatial component tensor $\discretized{w}_0\in\mathbb{R}^{N\times\vert\Delta_d\vert}$ given by the FE coefficients of the functions $w_{\iota(t)}(x)$, i.e.
  \begin{align*}
    w_{\iota(t)}(x) := \sum_{j\in[\vert\mathcal{T}\vert]}\discretized{w}_0[j,t]\varphi_j(x)
    \qquad\mbox{for all }t=0,\dots,\vert\Delta_d\vert-1,
  \end{align*}
  and stochastic cores $\discretized{w}_\ell\in\mathbb{R}^{\vert\Delta_d\vert\times d_\ell \times\vert\Delta_d\vert}$ for $\ell=1,\dots,L-1$ and $\discretized{w}_L\in\mathbb{R}^{\vert\Delta_d\vert\times d_L}$ given by
  \begin{align*}
    \discretized{w}_\ell[k_\ell,\mu_\ell,k_{\ell+1}] := \delta_{k_\ell k_{\ell+1}}\delta(k_\ell,\mu_\ell)
    \qquad\mbox{and}\qquad
    \discretized{w}_L[k_L,\mu_L] := \delta(k_L,\mu_L).
  \end{align*}
\end{proposition}

\begin{proof}%
  Contracting the last two cores $\discretized{w}_{L-1}$ and $\discretized{w}_L$ leads to
  \begin{equation*}
      \begin{split}
          &\sum_{k_L=0}^{\vert\Delta_d\vert-1} \discretized{w}_{L-1}[k_{L-1},\mu_{L-1},k_L]\discretized{w}_{L}[k_{L},\mu_{L}]\\
          &\qquad= \delta(k_{L-1},\mu_{L-1})\delta(k_{L-1},\mu_{L})\\
          &\qquad= \delta_{k_{L-1}0}^2\!\!\prod_{\ell=L-1}^{L}\!\! \delta_{\mu_{\ell}0} \ +\ (1-\delta_{k_{L-1}0})^2 \!\!\prod_{\ell=L-1}^{L}\!\! (\delta_{\mu_{\ell}0} + \delta_{k_{L-1}\iota_{\mu_{\ell}}}(1-\delta_{\mu_{\ell}0})).
      \end{split}
  \end{equation*}
  Iterating the contraction for the remaining stochastic cores $\discretized{w}_{\ell}$, $\ell=L-1,\dots,1$, yields the tensor $\discretized{T}\in\mathbb{R}^{\vert\Delta_d\vert\times d}$ given by
  \begin{align}
      \label{eq:supplement:exactTT_gamma:proof}
      \discretized{T}[k_1,\mu]
      = \delta_{k_1 0}^L\delta_{\mu 0} + (1-\delta_{k_1 0})^L\prod_{\ell=1}^{L} \Bigl( \delta_{\mu_{\ell} 0} + \delta_{k_{1}\iota_{\mu_{\ell}}} (1-\delta_{\mu_{\ell} 0}) \Bigr).
  \end{align}
  Since $\Delta_d$ is the set of univariate $L$-dimensional indices, for each $\mu\in\Lambda_d$ we have $\mu=\sum_{\ell=1}^{L} \mu_\ell e_\ell$ with $\mu_\ell e_\ell\in\Delta_d$.
  For any $\mu\in\Lambda_d\setminus\Delta_d$ there exist at least two $\ell_1\neq \ell_2 \in\{ 1,\dots,L \}$ with $\mu_{\ell_1},\mu_{\ell_2}>0$.
  Since $\iota^{-1}(\mu_{\ell_1}e_{\ell_1})\neq\iota^{-1}(\mu_{\ell_2}e_{\ell_2})$ and $0\in\Delta_d$, i.e.\ $\delta_{\mu 0}=0$, it follows that
  \begin{align*}
      \discretized{T}[k_1,\mu] 
      = \delta_{k_{1}\iota_{\mu_{\ell_1}}}\delta_{k_{1}\iota_{\mu_{\ell_2}}}(1-\delta_{k_1 0})^L\prod_{\ell_1,\ell_2\neq \ell=1}^{L} \Bigl( \delta_{\mu_{\ell} 0}(1-\delta_{k_{1}\iota_{\mu_{\ell}}}) + \delta_{k_{1}\iota_{\mu_{\ell}}} \Bigr)
      = 0.
  \end{align*}
  For any $\mu\in\Delta_d\setminus\{ 0 \}$ there exists exactly one $\ell\in\{ 1,\dots,L \}$ such that $\mu=\mu_\ell e_\ell$ for $\mu_\ell\in[1\sep d_\ell]$.
  Hence, $\delta_{\mu_m 0}=1$ for all $m=\{ 1,\dots,L \}$ with $m\neq \ell$ and~\eqref{eq:supplement:exactTT_gamma:proof} simplifies to $\discretized{T}[k_1,\mu] = (1-\delta_{k_1 0})\delta_{k_1\iota^{-1}(\mu)}$.
  Eventually, for $\mu=0$ we have $\iota_{\mu_\ell} = 0$ for all $\ell=1,\dots,L$ and thus $\discretized{T}[k_1,0] = \delta_{k_1 0} \delta_{k_1 \iota^{-1}(0)}$.
  Combining the cases above results in $\discretized{T}[k_1,\mu] = \delta_{k_1\iota^{-1}(\mu)}$ for all $\mu\in\Delta_d$.
  Combining this and the definition of $\discretized{w}_0$ with~\eqref{eq:discretization:univariate_function} concludes the proof.
\end{proof}

\subsubsection*{Lognormal coefficient field}

There exists no exact TT representation of the lognormal diffusion coefficient~\eqref{eq:setup:kappa}.
Several methods to obtain an approximation of~\eqref{eq:setup:kappa} in the TT format have been investigated~\cite{EMPS20,EHLMW14,DKLM15,DS19,EFHT21}.
Since Proposition~\ref{pro:discretization:exactTT_gamma} yields an exact representation of the affine exponent, the method of our choice is to compute the exponential of $\gamma-\gamma_0$ as proposed in~\cite{EFHT21}.
In this work, the authors use that the lognormal diffusion coefficient~\eqref{eq:setup:kappa} constitutes a holonomic function, i.e., $\kappa$ is the unique solution of the gradient system
\begin{equation}
  \label{eq:discretization:gradient_system}
  \begin{aligned}
    \grad_y \kappa(x,y) &:= \kappa(x,y)\grad_y (\gamma(x,y) - \gamma_0(x)), \\
    \kappa(x,y_0) &:= \exp(\gamma(x,y_0) - \gamma_0(x)),
  \end{aligned}
\end{equation}
for some arbitrary $y_0\in\Gamma_\kappa$ and all $x\in D$.
They show that a Galerkin approach for~\eqref{eq:discretization:gradient_system} emits a unique solution, which can be approximated efficiently in the TT format using the \emph{Alternating Linear Scheme} (ALS)~\cite{HRS12}.
Additionally, this approach yields a reliable and efficient error bound for the energy error induced by~\eqref{eq:discretization:gradient_system}, which does not involve any computational overhead.

\begin{remark}
  We utilize the structure of $\gamma$, $\kappa$ and $f$ to generate efficient low-rank approximations.
  However, it should be pointed out that there is no restriction per se and the used approximation technique can readily be applied to more complicated problems or alternative approximation techniques can be used as a substitute without further adaptation of other parts of our approach.
  As an example, one could consider non-intrusive reconstruction techniques such as a TT cross approximation~\cite{Oseledets2010} or a VMC reconstruction if the diffusion field is only accessible by pointwise evaluations.
\end{remark}

\section{Error estimation}%
\label{sec:estimator}

In this section we recall the residual based error estimator presented in~\cite{EigelMarschall2020lognormal}, which is an adaptation of the development in~\cite{EigelGittelson2014asgfem,EigelGittelson2014convergence}.
Additionally, we motivate an heuristic indicator to steer the number of regression samples in the adaptive algorithm.
The results are stated with the lognormal case~\ref{case:setup:lognormal} in mind, but equally hold true without any adaptation for the simpler affine coefficient, using the notation described in case~\ref{case:setup:affine}.

\subsection{Residual based error estimator}

We note that there exists no exact TT representation of the lognormal diffusion coefficient, thus the approximation of $\kappa$ described in Section~\ref{sec:discretization:diff_coeff} introduces an additional error.
Since this error can be controlled independently, we assume the approximation of $\kappa$ to be sufficiently accurate such that it can be neglected henceforth.

In the following we assume some fixed FE polynomial degree $p\in\mathbb{N}$ and consider $M \leq L \in \mathbb{N}$.
Furthermore, let $d,q\in\mathcal{F}$ with $\operatorname{supp}(d) = \{ 1,\dots,M \} \subseteq \{ 1,\dots,L \} = \operatorname{supp}(q)$.
Assume $w_N\in \mathcal{V}_N=\mathcal{V}_N(\Lambda_d;\mathcal{T},p)$ is given by the TT representation
\begin{align*}
    w_N(x,y) := \sum_{j\in[N]}  \sum_{\mu\in\Lambda_d} \Bigl( \sum_{k=1}^{r} \discretized{w}_0[j,k_1] \prod_{m=1}^{M}\discretized{w}_m[k_m,\mu_m,k_{m+1}]\Bigr) \varphi_j(x) P_\mu(y)
\end{align*}
with ranks $r\in\mathbb{N}^M$.
Similarly, assume that the coefficient field $a_N\in L^2(\Gamma,\pi;L^{\infty}(D))$ is given in a semi-discretized form
\begin{align*}
  a_N(x,y) := \sum_{\nu\in\Lambda_q} \Bigl( \sum_{k=1}^{s} a_0[k_1](x) \prod_{\ell=1}^{L}\discretized{a}_\ell[k_\ell,\nu_\ell,k_{\ell+1}]\Bigr) P_\nu(y)
\end{align*}
with ranks $s\in\mathbb{N}^L$ and $a_0[k_1]\in\mathcal{X}$ for all $k_1=1,\dots,s_1$.
Define the residual $\mathcal{R}(v)\in \mathcal{V}^* = L^2(\Gamma,\pi;\mathcal{X}^*)$ of~\eqref{eq:setup:varriational_form} by $\mathcal{R}(v) := F - B(v,\bullet)$.
The energy error can then be bounded in the following way.

\begin{theorem}[{{\cite[Theorem $5.1$]{EGSZ14}}}]%
  \label{thm:estimator:EGSZ14}
  Let $\mathcal{V}_N\subset\mathcal{V}$ be a closed subspace and $w_N\in\mathcal{V}_N$ arbitrary.
  Let $u_N$ denote the Galerkin projection with respect to $B$ of $u$ onto $\mathcal{V}_N$.
  It then holds that
  \begin{equation*}
    \Vert u-w_N\Vert_B^2 
    \leq \left( \sup_{v\in\mathcal{V}\setminus\{0\}} \frac{\vert \langle \mathcal{R}(w_N), (1-I_{\mathrm{C}}) v\rangle_{\mathcal{V}^*,\mathcal{V}}\vert}{\check{c}\,\Vert \grad v\Vert_{\pi_0,D}} + c_I \Vert u_N - w_N\Vert_B \right)^2 + \Vert u_N - w_N\Vert_B^2.
  \end{equation*}
  Here, $I_\mathrm{C}$ denotes the tensor product interpolation operator defined in~\cite{EigelMarschall2020lognormal}, $c_I$ is the operator norm of $1-I_{\mathrm{C}}$ with respect to $\Vert\bullet\Vert_B$ and $\check{c}$ is the coercivity constant of the bilinear form $B$.
\end{theorem}

We point out, that showing coercivity of the bilinear form for the lognormal case~\ref{case:setup:lognormal} is not trivial, see Supplement~\ref{supplement:well_posedness_lognormal} for more details and references.
Since we consider $w_N\in\mathcal{V}_N$ as well as a finite expansion $a_N$ of the diffusion coefficient $a$, the residual $\mathcal{R}(w_N)\in\mathcal{V}_N^*:=L^2(\Gamma,\pi;\mathcal{X}_N^*)$ is itself characterized by a finite polynomial expansion.
In particular, we have 
\begin{equation*}
  \langle\mathcal{R}(w_N),v_N\rangle_{\mathcal{V}_N^*, \mathcal{V}_N} = \int_\Gamma \int_D f v_N - r(w_N)\cdot\grad v_N\,\mathrm{d}x\mathrm{d}\pi(y)
\end{equation*}
for $r(w_N) := a_N\grad w_N$.
The semi-discrete expansion
\begin{align}
  \label{eq:estimator:expansion_r}
  r(w_N)(x,y)
  = \sum_{\mu\in\Lambda_{d+q-1}} r_\mu(w_N)(x) P_\mu(y),
\end{align}
with $r_\mu(w_N)$ derived from the TT representations of $a_N$ and $w_N$, allows to split the residual into an active and an inactive part.
With the inactive set $\Delta := \Lambda_{(d+q-1)}\setminus \ \Lambda_{d}$, consider the splitting $\mathcal{R}(w_N) = \mathcal{R}_{\Lambda_d}(w_N) + \mathcal{R}_{\Delta}(w_N)$ for
\begin{align*}
  \mathcal{R}_{\Lambda_d}(w_N) := f + \sum_{\mu\in\Lambda_d} \ddiv(r_{\mu}(w_N))P_\mu
  \qquad\mbox{and}\qquad
  \mathcal{R}_{\Delta}(w_N) := \sum_{\mu\in\Delta} \ddiv(r_{\mu}(w_N))P_\mu.
\end{align*}

\paragraph*{Deterministic estimator contributions}

The active part of the residual is associated with the deterministic approximation error of the FE discretization.
This enables to estimate the error on each triangle for the active set.
Define for any $w_N\in\mathcal{V}_N$ the deterministic error estimator contribution
\begin{align}
  \label{eq:estimator:eta_det}
  \eta_\mathrm{det}(w_N,\mathcal{T},\Lambda_d)^2
  := \sum_{T\in\mathcal{T}} \eta_{\mathrm{det},T}(w_N,\Lambda_d)^2 + \sum_{E\in\mathcal{E}} \eta_{\mathrm{det},E}(w_N,\Lambda_d)^2,
\end{align}
where the volume and edge terms respectively read 
\begin{align}
  \label{eq:estimator:eta_det:vol}
  \eta_{\mathrm{det},T}(w_N,\Lambda_d) &:= h_T \Vert \mathcal{R}_{\Lambda_d}(w_N)\,\zeta_{\vartheta\rho}\Vert_{\pi_0,T} &\mbox{for all }T\in\mathcal{T},\\
  \label{eq:estimator:eta_det:jump}
  \eta_{\mathrm{det},E}(w_N,\Lambda_d) &:= h_E^{1/2} \Vert \sum_{\mu\in\Lambda_d}\llbracket r_\mu(w_N) \rrbracket_E P_\mu\,\zeta_{\vartheta\rho}\Vert_{\pi_0,E} &\mbox{for all }E\in\mathcal{E}.
\end{align}
The deterministic estimator contribution bounds the active part of the residual as the following lemma shows.

\begin{lemma}[{{\cite[Proposition $5.3$]{EigelMarschall2020lognormal}}}]%
  \label{lem:estimator:est_det}
  For any $v\in\mathcal{V}$ and any $w_N\in\mathcal{V}_N$, it holds that
  \begin{align*}
    \vert \langle \mathcal{R}_{\Lambda_d}(w_N), (1-I_{\mathrm{C}})v\rangle_{\mathcal{V}^*,\mathcal{V}}\vert
    \leq c_\mathrm{det} \eta_\mathrm{det}(w_N,\mathcal{T},\Lambda_d) \Vert \grad v\Vert_{\pi_0,D}.
  \end{align*}
\end{lemma}

\paragraph*{Stochastic estimator contributions}

The inactive part of the residual is associated with the approximation error due to the truncation of the polynomial expansion.
This can be used to obtain a bound for the residual on the inactive part.
Define for any $w_N\in\mathcal{V}_N$ the stochastic error estimator
\begin{equation}
  \label{eq:estimator:eta_sto}
    \eta_\mathrm{sto}(w_N,\Delta)  := \Vert \sum_{\mu\in\Delta} r_\mu(w_N) P_\mu \zeta_{\vartheta\rho}\Vert_{\pi_0,D}.
\end{equation}
Again, the stochastic error is bounded for the affine as well as the lognormal case.

\begin{lemma}[{{\cite[Proposition $5.5$]{EigelMarschall2020lognormal}}}]%
  \label{lem:estimator:est_sto}
  For any $v\in\mathcal{V}$ and any $w_N\in\mathcal{V}_N$, it holds that
  \begin{align*}
    \vert \langle \mathcal{R}_{\Delta}(w_N), (1-I_{\mathrm{C}})v\rangle_{\mathcal{V}^*,\mathcal{V}}\vert
    \leq \eta_\mathrm{sto}(w_N,\Delta)\, \Vert \grad v \Vert_{\pi_0,D}.
  \end{align*}
\end{lemma}

To obtain a localization that can be used in the adaptive refinement strategy, we split $\eta_{\mathrm{sto}}$ into different parts each providing information about the influence of the individual modes.
For all $\ell=1,\dots,L$ and \emph{look-ahead} $t_\ell\in[1\sep q_\ell-1]$, consider the index sets of uncoupled parameters
\begin{align*}
  \Delta_{\ell,t_\ell} 
  := \bigotimes_{m=1}^{\ell-1}[d_m] \ \otimes\ [d_\ell:d_\ell+t_\ell] \ \otimes\ \bigotimes_{m=\ell+1}^{L}[d_m] \ \otimes [1] \dots
\end{align*}
The look ahead $t_\ell$ allows to consider more information about the behaviour of the stochastic dimensions than the index sets $\Xi_\ell$ considered in~\cite{EigelMarschall2020lognormal}.
The local stochastic estimator contributions $\eta_{\mathrm{sto}}^2(w_N,\Delta_{\ell,t_\ell})$ are then defined as in~\eqref{eq:estimator:eta_sto} for the sets $\Delta_{\ell,t_\ell}$.

\paragraph*{Algebraic estimator contributions}

The algebraic error $\Vert u_N - w_N\Vert_B$ incorporates the distance of $w_N$ to the $\mathcal{V}_N$ best approximation $u_N$.
Since we employ the sample based VMC regression to obtain an approximation $u_{N,r,n}$ of $u_N$, this error can be used as an indicator to control the number of VMC samples to guarantee that $u_{N,r,n}$ is a quasi-best approximation with high probability.
The algebraic error can be bounded by the quantity
\begin{align}
  \label{eq:estimator:eta_alg}
  \eta_\mathrm{alg}(w_N) 
  := \Vert (\discretized{B}\discretized{w}-\discretized{f}) \discretized{S}^{-1/2}\Vert_2,
\end{align}
where $\Vert \bullet \Vert_2$ denotes the Frobenius norm, $\discretized{f}$ is the coefficient tensor of the right-hand side $f$ in $\mathcal{V}_N$ and $\discretized{B}$ is the discrete version of the operator induced by~\eqref{eq:setup:B}.
The rank-one base change tensor $\discretized{S}$ translates integrals of Hermite polynomials with respect to the measure $\pi$ to $\pi_0$ and is given by the components $\discretized{S}_0[i,j] := \int_{D} \grad\varphi_i\cdot\grad\varphi_j \,\mathrm{d}x$ for the spatial dimension and
\begin{align}
  \label{eq:estimator:base_change}
  \discretized{S}_m[\mu_m,\mu_m']
  := \int_{\Gamma_m} P_{\mu_m}^m(y_m) P_{_{\mu_m'}}^m(y_m) \,\mathrm{d}\pi_{m}(y_m;\sigma_m(0))
\end{align}
for each mode.
In the affine case~\ref{case:setup:affine} this implies $\discretized{S}_m = I_{d_m}$.
The algebraic error is bounded in the following way.

\begin{lemma}[{{\cite[Proposition $5.6$]{EigelMarschall2020lognormal}}}]%
  \label{lem:estimator:est_alg}
  For any $w_N \in \mathcal{V}_N$ and the Galerkin solution $u_N\in\mathcal{V}_N$ of~\eqref{eq:setup:galerkin_system}, it holds that
  \begin{align*}
    \Vert u_N - w_N \Vert_B \leq \check{c}^{-1} \eta_\mathrm{alg}(w_N).
  \end{align*}
\end{lemma}

\paragraph*{Combined error estimator}

As a corollary of Theorem~\ref{thm:estimator:EGSZ14} and Lemmas~\ref{lem:estimator:est_det}--\ref{lem:estimator:est_alg}, the energy error can be bounded by the combined overall error estimator
\begin{align}
  \label{eq:estimator:est_total}
  \eta(w_N)^2
  := \bigl( c_\mathrm{det}\eta_\mathrm{det}(w_N,\mathcal{T},\Lambda_d) + \eta_\mathrm{sto}(w_N,\Delta) + c_I\eta_\mathrm{alg}(w_N) \bigr)^2 + \eta_\mathrm{alg}(w_N)^2.
\end{align}

\begin{corollary}[{{\cite[Corollary $5.7$]{EigelMarschall2020lognormal}}}]%
  \label{cor:estimator:est_total}
  For any $w_N\in\mathcal{V}_N$ it holds $\Vert u-w_N\Vert_B \leq \check{c}^{-1}\,\eta(w_N)$.
\end{corollary}

\subsection{Regression error indicator}

Since the proposed method fundamentally relies on solving the nonlinear least squares problem~\eqref{eq:discretization:empirical_min} to compute a tensor representation of the solution, it is important to adapt the number of samples in order to guarantee a robust approximation.
In this section we discuss a heuristic indicator for doing this when~\eqref{eq:discretization:empirical_min} is solved by means of an ALS method.
ALS methods are part of a family of iterative methods that minimize~\eqref{eq:discretization:empirical_min} by solving a sequence of least squares problems
\begin{equation}%
    \label{eq:estimator:empirical_min_linear}
    w_{N,k,n} := \argmin_{w_N\in\mathcal{W}_{k}} \Vert u_N - w_N\Vert_n^2
\end{equation}
on linear subsets $\mathcal{W}_k\subseteq\mathcal{M}_r$ for $k\in\mathbb{N}$.
Note that $w_{N,k,n}$ is an empirical estimate of the best approximation $w_{N,k}$ of $u_N$ in the linear space $\mathcal{W}_k$.
The ensuing estimation error can be bounded by the subsequent lemma that is adapted from~\cite{CohenMigliorati2017}.

\begin{lemma}%
    \label{lem:estimator:Gramian_error_bound}
    Let $\{P_1,\ldots,P_D\}$ be any orthonormal basis of $\mathcal{W}_k$ and let $G_n\in \mathbb{R}^{D\times D}$ be the empirical Gramian given by $(G_n)_{ij} := \frac{1}{n}\sum_{t=1}^n P_i(y^{(t)}) P_j(y^{(t)})$.
    If the smallest eigenvalue of $G_n$ satisfies $\lambda_{\min}(G_n) > 0$, then 
    \begin{equation*}
        \Vert u_N - w_{N,k,n}\Vert_{L^2(\Gamma,\pi_{\rho};\mathcal{X})}
        \le (1+\lambda_{\mathrm{min}}(G_n)^{-1/2})\Vert u_N - w_{N,k}\Vert_{L^\infty(\Gamma,\pi_{\rho};\mathcal{X})} .
    \end{equation*}
\end{lemma}

\begin{proof}%
  In the following we abbreviate $\Vert\bullet\Vert_{\infty,D} := \Vert\bullet\Vert_{L^\infty(\Gamma,\pi_{\rho};L^2(D))}$.
  Let $v\in\mathcal{W}_k$ be arbitrary and let $\discretized{v}$ denote the coefficients of $v$ with respect to the basis $\{P_1,\ldots,P_D\}$.
  Now note that $\Vert v\Vert_n^2 = \discretized{v}^\intercal G_n \discretized{v} \ge \lambda_{\mathrm{min}}(G_n) \Vert\grad_x v\Vert_{\pi_\rho,D}^2$ is a Hilbert space norm and that $w_{N,k,n}$ is the orthogonal projection of $u_N$ onto the space $\mathcal{W}_k$ with respect to the corresponding inner product.
  Hence,
  \begin{equation}
      \Vert u_N-v\Vert_n^2 = \Vert u_N-w_{N,k,n}\Vert_n^2 + \Vert w_{N,k,n}-v\Vert_n^2 \ge \Vert w_{N,k,n}-v\Vert_n^2 .
  \end{equation}
  Combining this observation with the lower bound $\Vert \grad_x v\Vert_{\pi_\rho,D} \le \lambda_{\mathrm{min}}(G_n)^{-1/2} \Vert v\Vert_n$ yields
  \begin{align*}
      \Vert\grad_x(u_N-w_{N,k,n})\Vert_{\pi_\rho,D}
      &\le \Vert\grad_x(u_N-v)\Vert_{\pi_\rho,D} + \lambda_{\mathrm{min}}(G_n)^{-1/2}\Vert v-w_{N,k,n}\Vert_n \\
      &\le \Vert\grad_x(u_N-v)\Vert_{\pi_\rho,D} + \lambda_{\mathrm{min}}(G_n)^{-1/2}\Vert u_N-v\Vert_n \\
      &\le (1+\lambda_{\mathrm{min}}(G_n)^{-1/2})\Vert\grad_x(u_N-v)\Vert_{\infty,D}.
  \end{align*}
  The first inequality is simply the triangle inequality and the final inequality follows from the fact that both $\Vert\grad_x \bullet\Vert_{\pi_\rho,D}$ and $\Vert \bullet\Vert_n$ are dominated by $\Vert\grad_x \bullet\Vert_{\infty,D}$.
  Since $v$ is arbitrary, we can substitute $v=w_{N,k}$ to prove Lemma~\ref{lem:estimator:Gramian_error_bound}.
\end{proof}

Assuming $\lambda_{\mathrm{min}}(G_n) \geq \varepsilon_G > 0$, Lemma~\ref{lem:estimator:Gramian_error_bound} states that the error of the least squares approximation can be bounded up to the constant factor of $(1+\varepsilon_G^{-1/2})$ by the best approximation error.
As $G_n$ converges towards the identity, there always exists $n\in\mathbb{N}$ such that $\lambda_{\mathrm{min}}(G_n)$ exceeds the given threshold $\varepsilon_G$.
Hence the condition $\lambda_{\mathrm{min}}(G_n) \geq \varepsilon_G$ can be used as an indicator for the necessity of increasing the number of samples.

\begin{remark}%
  \label{rem:estimator:gramian_tt}
  A full representation of $G_n$ is not available due to the curse of dimensionality.
  It is, however, possible to assemble the Gramians $G_{n,k}$ in the linear subspaces $\mathcal{W}_k\subseteq\mathcal{M}_r$ by fixing and contracting each but the $k$-th core of the current solution $w_{N,k,n}$ to the rank-one compression of $G_n$ during each step of the ALS method.
  Doing this for each component $k=1,\dots,M$ thus yields the criterion
  \begin{equation*}
    0 < \varepsilon_G \leq \min_{k=1,\dots,M} \lambda_{\mathrm{min}}(G_{n,k}).
  \end{equation*}
  Even though it is possible to do so in each step of the ALS method, this incurs a large computational overhead and requires extensive intervention with existing code.
  Hence we regard the optimization algorithm as a black box and compute $G_{n,k}$ for $k=1,\dots,M$ only once for the final result $u_{N,r,n}$.
\end{remark}

We note that the heuristic approach of using $u_{N,r,n}$ described in Remark~\ref{rem:estimator:gramian_tt} may be unrelated to the true Gramians to the local solutions $w_{N,k,n}$.
Thus the following algorithm cannot guarantee that sufficiently many samples are available in each step of the optimization.
However, we did not observe any problems with this heuristic in our experiments.

\medskip
\noindent
\begin{minipage}{\linewidth}
\begin{algorithm2e}[H]\label{alg:resample}
  \caption{Increase number of VMC samples ($\mathtt{update\_samples}$)}
    \KwInput{%
        TT-representation of $u_{N,r,n}$;
        initial samples $\{y^{(i)}\}_{i=1}^{n}$;
        stochastic dimensions $d$;
        ratio $\theta_{\mathrm{alg}}>0$;
        minimum eigenvalue threshold $\varepsilon_G$
    }
    assemble polynomial tensor basis $\{P_1,\dots,P_D\}$ from dimension tuple $d$\;
    compute $G_{n,k}$ for $k=1,\dots,M$ using $u_{N,r,n}$ as in Remark~\eqref{rem:estimator:gramian_tt}\;
    \While{$\min_k\lambda_{\mathrm{min}}(G_{n,k}) < \varepsilon_G$}{
      draw $\lfloor\theta_{\mathrm{alg}}n\rfloor$ new samples $\{y_{\mathrm{new}}^{(j)}\}$\;
      $\{y^{(i)}\} \gets \{y^{(i)}\} \cup \{y_{\mathrm{new}}^{(j)}\}$\;
        update $G_{n,k}$ for $k=1,\dots,M$\;
    }
    \Return{samples $\{y^{(i)}\}$}
\end{algorithm2e}
\end{minipage}

\section{Algorithmic realization}%
\label{sec:algorithm}

This section describes the algorithm that steers the adaptive refinement of the FE space, the number of active modes as well as the VMC sampling and compression errors.
Given the approximation $u_{N,r,n}\in\mathcal{V}_N$ of the solution obtained via the VMC approach and the diffusion coefficient $a_N\in\mathcal{V}_N(\Lambda_q;\mathcal{T},p)$ in TT format, the estimator contributions can be computed efficiently as detailed in~\cite{EigelMarschall2020lognormal}.

\medskip
\noindent
\begin{minipage}{\linewidth}
\begin{algorithm2e}[H]\label{alg:mark_and_refine}
  \caption{Mark \& Refine ($\mathtt{mark\_and\_refine}$)}
    \KwInput{
        mesh $\mathcal{T}$;
        stochastic dimensions $d$;
        samples $\{y^{(i)}\}$;
        solution $u_{N,r,n}$;
        global estimator contributions $\eta_{\mathrm{det}}(u_{N,r,n},\mathcal{T},\Lambda_d)$, $\eta_{\mathrm{sto}}(u_{N,r,n},\Delta)$, $\eta_{\mathrm{alg}}(u_{N,r,n})$;
        local estimator contributions $\eta_{\mathrm{det},T}(u_{N,r,n},\Lambda_d)$, $\eta_{\mathrm{det},E}(u_{N,r,n},\Lambda_d)$, $\eta_{\mathrm{sto}}(u_{N,r,n},\Delta_{m,t_m})$;
        D\"orfler thresholds $0 < \theta_{\mathrm{det}}, \theta_{\mathrm{sto}} \leq 1$;
        ratio $\theta_{\mathrm{alg}}>0$;
        minimal eigenvalue threshold $\varepsilon_G$
    }
    set $\hat\eta = \max\{ \eta_{\mathrm{det}}(u_{N,r,n},\mathcal{T},\Lambda_d), \eta_{\mathrm{sto}}(u_{N,r,n},\Delta), \eta_{\mathrm{alg}}(u_{N,r,n})\}$;\\
    \If{$\eta_{\mathrm{det}}(u_{N,r,n},\mathcal{T},\Lambda_d) = \hat\eta$}{
      choose minimal set $\mathcal{M}\subset\mathcal{T}$, such that $\eta_{\mathrm{det}}(u_{N,r,n},\mathcal{M},\Lambda_d) \geq \theta_{\mathrm{det}} \eta_{\mathrm{det}}(u_{N,r,n},\mathcal{T},\Lambda_d)$;\\
      $\mathcal{T}\gets\mathtt{bisect}(\mathcal{T},\mathcal{M})$;
    }
    \ElseIf{$\eta_{\mathrm{sto}}(u_{N,r,n},\Delta) = \hat\eta$}{
      choose minimal set $\mathcal{M}\subset\{ 1,\dots,M+1 \}$, such that $\sum_{m\in\mathcal{M}} \eta_{\mathrm{sto}}(u_{N,r,n};\Delta_{m,t_m}) \geq \theta_{\mathrm{sto}} \eta_{\mathrm{sto}}(u_{N,r,n},\Delta)$;\\
      $d_m \gets d_m+1$ for $m\in\mathcal{M}$;\\
    }
    \Else{\label{alg:mark_and_refine:eta_alg_ref:start}
      compute $G_{n,k}$ for $k=1,\dots,M$ as in Remark~\eqref{rem:estimator:gramian_tt};\\
      \If{$\min_{k}\lambda_{\min}(G_{n,k}) < \varepsilon_G$}{\label{alg:mark_and_refine:eta_alg_ref}
        $\{y^{(i)}\} \gets \mathtt{update\_samples}(\{y^{(i)}\}, u_{N,r,n}, d, \theta_{\mathrm{alg}}, \varepsilon_G)$
      }
      \Else{
        choose minimal set $\mathcal{M}\subset\{ 1,\dots,M+1 \}$, such that $\sum_{m\in\mathcal{M}} \eta_{\mathrm{sto}}(u_{N,r,n};\Delta_{m,t_m}) \geq \theta_{\mathrm{sto}} \eta_{\mathrm{sto}}(u_{N,r,n},\Delta)$;\\
        $d_m \gets d_m+1$ for $m\in\mathcal{M}$;\label{alg:mark_and_refine:eta_alg_ref:end}\\
      }
    }
    \Return{$\mathcal{T}$, $d$, $\{y^{(i)}\}$}
\end{algorithm2e}
\end{minipage}
\medskip

With the combined error estimator from Corollary~\ref{cor:estimator:est_total} it is possible to define an algorithm that automatically steers the refinement of the spatial triangulation $\mathcal{T}$, activates modes or increases stochastic dimensions if necessary and ensures that the low-rank representation error of the solution $u_{N,r,n}$ is not predominant.
The adaptive algorithm described in this section is based on the one presented in~\cite{EGSZ14,EGSZ15,EMPS20} and consists in principle of three major steps, namely \emph{SOLVE}, \emph{ESTIMATE} and \emph{MARK \& REFINE}, which are iterated until some stopping condition is satisfied.

Given some triangulation $\mathcal{T}$, a fixed FE polynomial degree $p\in\mathbb{N}_0$ and $d,q\in\mathcal{F}$ with $\operatorname{supp}(d)\subset\operatorname{supp}(q)$, the \emph{SOLVE} step generates low-rank approximations of the solution $u_{N,r,n}\in\mathcal{V}_N(\Lambda_d;\mathcal{T},p)$, the diffusion coefficient $a_N\in\mathcal{V}_N(\Lambda_q;\mathcal{T},p)$ and the right-hand side $f\in\mathcal{V}_N(\Lambda_d;\mathcal{T},p=0)$.
The solution $u_N$ is reconstructed from $N_\textrm{VMC}$ many training samples $(y^{(i)}, u^{(i)})$ using the VMC method, which, in our case, chooses the rank of the solution adaptively by an alternating directional fitting algorithm~\cite{EigelNeumann2019,xerus}.
To obtain an approximation of the diffusion coefficient, we use Proposition~\ref{pro:discretization:exactTT_gamma} to construct the exact representation of the affine field~\eqref{eq:setup:gamma}.
Then, depending on the coefficient type, we either use the exact TT representation directly or compute the Galerkin projection of~\eqref{eq:discretization:gradient_system} via the $\mathtt{ExpTT}$ algorithm presented in~\cite[Algorithm $2$]{EFHT21}.
When employing the \texttt{ExpTT} algorithm, we ensure that the error is sufficiently small by automatically decreasing the rounding threshold and increasing the projection dimensions during rescaling until a certain tolerance is reached.
Since we assume a right-hand side independent of $y$, Proposition~\ref{pro:discretization:exactTT_gamma} also yields an exact representation of $f$ in the TT format.
We note that the TT representations $u_{N,r,n}$, $a_N$ and $f$ in case~\ref{case:setup:lognormal} are not with respect to the correctly scaled Hermite basis.
Thus a subsequent transformation to the correct basis has to be performed.

The local and global estimator contributions according to Section~\ref{sec:estimator} are computed in the \emph{ESTIMATE} step.
We emphasize that the global contributions $\eta_{\mathrm{det}}(u_{N,r,n},\mathcal{T},\Lambda_d)$ and $\eta_{\mathrm{sto}}(u_{N,r,n},\Delta)$ are used to determine the refinement strategy, but that we use localized versions to determine which triangles and modes are refined.
The localized contributions $\eta_{\mathrm{det}}(u_{N,r,n},\{T\},\Lambda_d)$ and $\eta_{\mathrm{sto}}(u_{N,r,n},\Delta_{\ell,t_\ell})$ are defined for each triangle $T\in\mathcal{T}$ and each mode $\ell=1,\dots,M+1$, respectively, where $M$ is the number of active modes of the solution $u_{N,r,n}$.

The \emph{MARK \& REFINE} step uses the estimator contributions in combination with different selection criteria to enlarge the FE space, the stochastic space or the number of training samples for the VMC algorithm, respectively.
We employ a D\"orfler marking strategy for the spatial and stochastic refinement and a simple relative increase in the number of VMC samples as described in Algorithm~\ref{alg:mark_and_refine}.
For the refinement of the spatial mesh $\mathcal{T}$ we use newest vertex bisection~\cite{Stevenson2008} on all marked elements $T\in\mathcal{M}$, which is denoted by $\mathtt{bisect}(\mathcal{T},\mathcal{M})$ in Algorithm~\ref{alg:mark_and_refine}.

\begin{remark}%
  Note that Algorithm~\ref{alg:mark_and_refine} distinguishes between enlarging the stochastic space and increasing the number of training samples if the algebraic estimator contribution dominates (ll.~\ref{alg:mark_and_refine:eta_alg_ref:start}--\ref{alg:mark_and_refine:eta_alg_ref:end}).
  Typically $\eta_{\mathrm{alg}}$ is associated with the approximation error caused by the low-rank compression~\cite{EigelMarschall2020lognormal}.
  As our recovery method is rank-adaptive, this implies that an increase of the number of training samples would be required.
  We observe in our experiments, however, that increasing $N_{\mathrm{VMC}}$ does not always lead to a reduction of $\eta_{\mathrm{alg}}$, but that enlarging the stochastic space is necessary if $\min_{k} \lambda_{\mathrm{min}}(G_{n,k})$ is already large enough.
  Since $\eta_{\mathrm{alg}}$ is essentially the residual of~\eqref{eq:setup:galerkin_system}, it can be expected to be dependent on the compression error $\Vert u_N-u_{N,r,n}\Vert_B$ as well as errors caused by finite deterministic and stochastic spaces, which explains this behaviour.
  We also note that refining the spatial domain has no significant influence on the magnitude of $\eta_{\mathrm{alg}}$ in the experiments.
\end{remark}

\medskip
\noindent
\begin{minipage}{\linewidth}
\begin{algorithm2e}[H]\label{alg:aVMC}
  \caption{Adaptive non-intrusive algorithm ($\mathtt{aVMC}$)}
    \KwInput{
        initial mesh $\mathcal{T}_0$;
        FE polynomial degree $p$;
        initial stochastic dimensions $d$;
        initial samples $\{y^{(i)}\}$;
        D\"orfler thresholds $0 < \theta_{\mathrm{det}}, \theta_{\mathrm{sto}} \leq 1$;
        ratio $\theta_{alg}>0$;
        max.\ number of iterations $N_{\mathrm{iter}}$;
        max.\ number of TT-DoFs $N_{\mathrm{TT}}$;
        accuracy $\varepsilon$;
        minimal eigenvalue threshold $\varepsilon_G$;
    }
    \KwOutput{
        solution $u_{N,r,n}$;
        combined estimator $\eta(u_{N,r,n})$ for each iteration;
    }
    \For{$j = 1, \dots, N_{\mathrm{iter}}$}{
      \emph{SOLVE}\\
      \hspace{0.2in} generate training samples $\{ u^{(i)}=u(x,y^{(i)})\}$ for $i = 1,\dots,N_{\mathrm{VMC}}$;\\
      \hspace{0.2in} $u_{N,r,n} \gets \mathtt{VMC}\bigl(\mathcal{V}_N(\Lambda_d;\mathcal{T},p), \{y^{(i)}\}, \{u^{(i)}\}\bigr)$;\\ \label{alg:aVMC:VMC}
      \hspace{0.2in} construct $\gamma$ and $f$ according to Proposition~\ref{pro:discretization:exactTT_gamma};\\
      \hspace{0.2in} $a_N \gets \gamma$ in case~\ref{case:setup:affine} \hspace{0.1in}{\bf OR}\hspace{0.1in} $a_N \gets \mathtt{ExpTT}(\gamma, q, y_0=0, \mathrm{tol}=10^{-5})$ in case~\ref{case:setup:lognormal};\\\label{alg:aVMC:expTT}
      \emph{ESTIMATE}\\
      \hspace{0.2in} compute $\eta_{\mathrm{det}}(u_{N,r,n},\mathcal{T},\Lambda_d), \bigl(\eta_{\mathrm{det}}(u_{N,r,n},\{T\},\Lambda_d)\bigr)_{T\in\mathcal{T}}$ according to~\eqref{eq:estimator:eta_det}--\eqref{eq:estimator:eta_det:jump};\\
      \hspace{0.2in} compute $\eta_{\mathrm{sto}}(u_{N,r,n},\Delta), \bigl(\eta_{\mathrm{sto}}(u_{N,r,n};\Delta_{\ell,t_\ell})\bigr)_{\ell=1}^{M+1}$ according to~\eqref{eq:estimator:eta_sto};\\
      \hspace{0.2in} compute $\eta_{\mathrm{alg}}(u_{N,r,n})$ according to~\eqref{eq:estimator:eta_alg};\\
      \hspace{0.2in} compute $\eta_j = \eta_{j}(u_{N,r,n})$ according to Corollary~\ref{cor:estimator:est_total};\\
      \If{$\eta_{j}(u_{N,r,n}) < \varepsilon$ \hspace{0.05in}or\hspace{0.05in} $\operatorname{tt-dofs}(u_{N,r,n}) > N_{\mathrm{TT}}$ \hspace{0.05in}or\hspace{0.05in} $N_{\mathrm{VMC}} > 10^5$}{
        \textbf{break} 
      }
      \emph{MARK \& REFINE}\\
      \hspace{0.2in} $\eta_{\mathrm{glob}} := (\eta_{\mathrm{det}}(u_{N,r,n},\mathcal{T},\Lambda_d),\eta_{\mathrm{sto}}(u_{N,r,n},\Delta),\eta_{\mathrm{alg}}(u_{N,r,n}))$;\\
      \hspace{0.2in} $\eta_{\mathrm{loc}} := \Bigl(\bigl(\eta_{\mathrm{det}}(u_{N,r,n},\{T\},\Lambda_d)\bigr)_{T\in\mathcal{T}},\bigl(\eta_{\mathrm{sto}}(u_{N,r,n};\Delta_{\ell,t_\ell})\bigr)_{\ell=1}^{M+1}\Bigr)$;\\
      \hspace{0.2in} $\mathcal{T}, d, \{y^{(i)}\} \gets \mathtt{mark\_and\_refine}(\mathcal{T}, d, \{y^{(i)}\}, u_{N,r,n}, \eta_{\mathrm{glob}}, \eta_{\mathrm{loc}}, \theta_{\mathrm{det}}, \theta_{\mathrm{sto}}, \theta_{\mathrm{alg}},\varepsilon_G)$
    }
    \Return $u_{N,r,n}$, $(\eta_1, \eta_2, \dots)$
\end{algorithm2e}
\end{minipage}
\medskip

Algorithm~\ref{alg:aVMC} iterates the solve, estimate, mark and refine loop until the combined estimator $\eta(u_N)$ is sufficiently small or a maximum problem size is reached.
Note that the algorithmic realization of the VMC reconstruction is abbreviated by the $\mathtt{VMC}$ method in line~\ref{alg:aVMC:VMC}.

\section{Numerical experiments}%
\label{sec:experiments}

In this section we examine the performance of the adaptive algorithm for benchmark problems similar to~\cite{EPS17, EMPS20}.
Tensor calculus is carried out with the open source software package \texttt{xerus}~\cite{xerus}.
Finite element computations to generate training samples are conducted with the \texttt{FEniCS} package~\cite{fenics}.
As spatial domain we choose either the unit square $D=(0,1)^2$ or the L-shaped domain $D=(0,1)^2\setminus[0.5,1]^2$.
The derived total error estimator $\eta$ is used to steer the adaptive refinement of the triangulation $\mathcal{T}$, the space $\Lambda_d$ and the number of reconstruction samples $N_{\mathrm{VMC}}$.
We investigate the behaviour of the individual estimator contributions and how they influence the true (sampled) expected energy error.
Moreover, we comment on the complexity of the coefficient discretization.

\subsection{Computation of the error}

To validate the reliability of the estimator and its contributions in the adaptive scheme, we compute an empirical approximation of the true $L^2(\Gamma, \pi;\mathcal{X})$-error using $N_{\mathrm{MC}}$ samples, i.e.
\begin{align*}
  \mathcal{E}_u(u_{N,r,n})^2
  := \frac{1}{N_{\mathrm{MC}}} \sum_{i=1}^{N_{\mathrm{MC}}} \Vert u(y^{(i)}) - u_{N,r,n}(y^{(i)}) \Vert_{\mathcal{X}(\hat{\mathcal{T}})}^2
  \approx \Vert \grad(u - u_{N,r,n}) \Vert_{\pi,D}^2.
\end{align*}
Here, the parametric solution $u_{N,r,n}\in\mathcal{V}_N(\Lambda_d;\mathcal{T},p)$ is compared to the deterministic sampled solution $u(y^{(i)})$ projected onto a uniform refinement $\hat{\mathcal{T}}$ of the finest FE mesh obtained in the adaptive refinement loop.
Since all triangulations generated by Algorithm~\ref{alg:aVMC} as well as $\hat{\mathcal{T}}$ are nested, we employ simple nodal interpolation of each $u_{N,r,n}$ onto $\mathcal{\hat T}$ to guarantee $u_{N,r,n}\in\mathcal{V}_N(\Lambda_d; \hat{\mathcal{T}},p)$.
Note that the reference samples $u(y^{(i)})$ are computed by the black-box solver of the forward problem, i.e., no functional approximation of $u$ on $\hat{\mathcal{T}}$ is required to compute $\mathcal{E}_u(u_{N,r,n})$.
All estimator contributions depend on the diffusion coefficient and the right-hand side.
To guarantee that the low-rank approximation errors of $a$ and $f$ have a negligible impact, we monitor the relative empirical $L^2$-$L^\infty$ error given by
\begin{align*}
  \varepsilon_w^\infty(w_N)^2 := \frac{\sum_{i=1}^{N_{\mathrm{MC}}} \Vert w(y^{(i)}) - w_N(y^{(i)}) \Vert_{L^\infty(D)}^2}{\sum_{i=1}^{N_{\mathrm{MC}}} \Vert w(y^{(i)})) \Vert_{L^\infty(D)}^2},
  \qquad\mbox{for all } w, w_N\in L^2(\Gamma,\pi;L^\infty(D)).
\end{align*}
The choice of $N_{\mathrm{MC}}=250$ proved to be sufficient to obtain consistent estimates of the error in our experiments as well as in other works (cf.~\cite{EigelMarschall2020lognormal}).

\subsection{The stochastic model problem}%
\label{sec:experiments:model_problem}

In the numerical experiments, we consider the stationary diffusion problem~\eqref{eq:setup:darcy} with constant right-hand side $f(x,y) = 1$.
For the affine diffusion field, the expansion coefficients $\gamma_\ell$ enumerate planar Fourier modes in increasing total order and are given by $\gamma_0(x)=1$ and
\begin{align*}
  \gamma_\ell(x) = \frac{9}{10\zeta(\sigma)} \ell^{-\sigma} \cos\bigl(2\pi\beta_1(\ell) x_1\bigr)\, \cos\bigl(2\pi\beta_2(\ell) x_2\bigr),
  \qquad \ell=1,\dots,L,
\end{align*}
where $\zeta$ is the Riemann zeta function and, for $k(\ell) = \lfloor -\frac{1}{2} + \sqrt{\frac{1}{4} +2\ell} \rfloor$, $\beta_1(\ell) = \ell - k(\ell)(k(\ell)+1)/2$ and $\beta_2(\ell) = k(\ell) - \beta_1(\ell)$.
For our experiments we consider a slow ($\sigma=2$) and a fast ($\sigma=4$) decay rate to test if the adaptive algorithm captures the relevance of the higher order modes.
As we assume access to the diffusion coefficient only via a priori generated samples with a finite number of spatial observation points, we choose to discretize $a_N$ in the same finite element space as the solution, i.e.\ conforming Lagrange elements of order $p=1$ or $p=3$.
For the lognormal case~\ref{case:setup:lognormal}, we choose $\rho=1$ and $\vartheta=0.1$ similar to~\cite{EigelMarschall2020lognormal}.

\subsection{Tensor train representation of the diffusion coefficient}

The affine diffusion coefficient $\gamma$ and the constant right-hand side $f$ can be represented in TT format as described in Proposition~\ref{pro:discretization:exactTT_gamma}.

An approximation $\kappa_{N,s}$ of the lognormal coefficient field is computed via the Galerkin projection described in Section~\ref{sec:discretization:diff_coeff}.
In particular, we employ a scaling trick to account for numerical instabilities caused by the global polynomial approximation of an exponential function.
We choose a scaling $t\in\mathbb{N}$ and construct an approximation $\kappa_{N,s,t} \approx \exp(2^{-t}(\gamma-\gamma_0))$ via the Galerkin approach.
The diffusion coefficient $\kappa_{N,s}$ can then be recovered by squaring $\kappa_{N,s,t}$ $t$ times.
For a detailed description of the process we refer to the algorithms in~\cite{EFHT21}.

As scaling constant we choose $t=4$ and round $\kappa_{N,s,t}$ to a precision of $10^{-6}$ and project the stochastic space of each mode onto Hermite polynomials up to degree $10$ in each squaring step to reduce storage requirements.
The ALS algorithm computing the approximation $\kappa_{N,s,t}$ has a termination threshold of $10^{-8}$ that needs to be reached in the Frobenius norm of the difference of two successive iteration results.
This results in a relative approximation error $\varepsilon_{\kappa}^\infty(\kappa_{N,s})<10^{-4}$, which is at least one order of magnitude smaller then that of the solution $u_{N,r,n}$ in all experiments.

We also note that the exact TT representation of $\gamma$ has uniform ranks $s=L+1$, which results in large storage requirements throughout the computations.
Such large representation ranks are, however, not necessary for a sufficient numerical treatment, since the ranks decrease drastically upon rounding $\gamma$ to machine precision.
Reducing the ranks of $\gamma$ crucially decreases the required memory of both the estimator contributions and the Galerkin operator assembled during the \texttt{ExpTT} algorithm~\cite{EigelFarchmin2021expTT}, which is a limiting factor otherwise.
A detailed investigation of the approximation quality and representation ranks of $\kappa_{N,s}$ can be found in~\cite{EFHT21}.

\subsection{Adaptive convergence results}

The fully adaptive Algorithm~\ref{alg:aVMC} is instantiated with a single mode $M=1$ discretized with a linear polynomial, i.e.\ dimension $d_1=2$.
The initial spatial mesh consists of $\vert\mathcal{T}_0\vert=143$ triangles with $p=1$ ansatz functions and $\vert\mathcal{T}_0\vert=64$ with $p=3$ for the L-shaped domain.
For the initial triangulations of the unit square domain, we choose $\vert\mathcal{T}_0\vert=205$ triangles with $p=1$ ansatz functions and $\vert\mathcal{T}_0\vert=76$ with $p=3$.
The marking parameters are set to $\theta_{\mathrm{det}}=0.3$, $\theta_{\mathrm{sto}}=0.5$ and $\theta_{\mathrm{alg}}=0.3$, respectively,
and as threshold to increase the number of sampling points we choose $\varepsilon_G=10^{-3}$.
We start Algorithm~\ref{alg:aVMC} with $N_{\mathrm{VMC}}=100$ initial samples $\{y^{(i)}\}_{i=1}^{100}$ and choose $N_{\mathrm{iter}}=50$, $N_{\mathrm{TT}}=10^5$ and $\varepsilon=10^{-3}$ as stopping criteria for the algorithm.
Additionally, we terminate Algorithm~\ref{alg:aVMC} if the number of samples $N_{\mathrm{VMC}}$ exceeds $10^{5}$ as reconstruction times for the solution become prohibitive otherwise.

\begin{figure}[htp]
  \begin{center}
    \includegraphics[width=\textwidth]{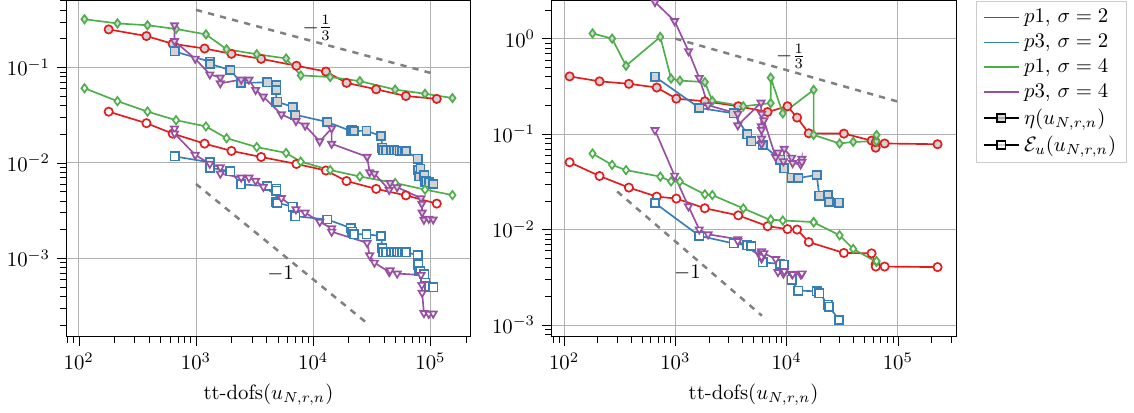}
  \end{center}
  \caption{%
    Sampled root mean square error $\mathcal{E}_u(u_{N,r,n})$ and total error estimator $\eta(u_{N,r,n})$ of the fully adaptive algorithm on the L-shaped domain.
    Considered are finite element approximations of order $p=1$ and $p=3$ for slow ($\sigma=2$) and fast ($\sigma=4$) decay for the affine (left) and lognormal (right) case.
  }%
  \label{fig:convergence_lshape}
\end{figure}

Figure~\ref{fig:convergence_lshape} depicts the true sampled root mean squared $H_0^1(D)$ error $\mathcal{E}_u(u_{N,r,n})$ and the corresponding overall error estimator $\eta(u_{N,r,n})$ for the affine and lognormal case, respectively.
Depicted are combinations of slow ($\sigma=2$) and fast ($\sigma=4$) decay rates and finite element discretization degrees $p=1$ and $p=3$ on the L-shaped domain.

We observe that there is no significant difference in the error and estimator magnitudes in the experiments between the two different computational domains.
For the results on the unit square domain, we refer to Supplement~\ref{supplement:unit_square}.
The experiments show that the deterministic estimator contribution $\eta_{\mathrm{det}}$ captures the singularity of the L-shaped domain and prioritizes to refine the mesh at the reentrant corner as known from deterministic adaptive FE methods.

The $p=3$ FE discretizations converge at roughly thrice the rate of the $p=1$ cases, which is expected.
The rates for both FE discretizations are the same as observed in~\cite{EGSZ14,BespalovPraetorius2019convergence,Bespalov2018,Bespalov2021twolevel} in the affine case.
For the lognormal case, to the knowledge of the authors only~\cite{EMPS20} has presented reliable error estimation yet.
Our results are again very similar to the results reported previously.

In both magnitude and convergence rate we observe almost no difference between the two decay rates.
The only exception to this is the $p=3$ FE discretization for the affine case, where the error and estimator for $\sigma=4$ reached slightly smaller values.

The obtained convergence rates for the respective affine and the lognormal experiments are the same, which indicates that the adaptive algorithm works robustly independent of the specific choice of the diffusion coefficient.
We observe that the overall magnitudes of error and estimator are slightly larger in the more involved lognormal case, but the estimator consistently overestimates the error by a factor of approximately $10$.
We also note that both error and estimator values are of the same order of magnitude with respect to the degrees of freedom as in~\cite{EGSZ14,BespalovPraetorius2019convergence,Bespalov2018,Bespalov2021twolevel,EMPS20}.

The most obvious difference between the two diffusion coefficient types is the smoothness of the numerical convergence graphs.
In the affine case the error and error estimator decrease monotonously, whereas they might sometimes increase after refinement in the lognormal case~\ref{case:setup:lognormal}.
Interestingly, this is most prevalent when considering a fast decay rate of $\sigma=4$.
Nevertheless, this behaviour is expected since the reconstruction accuracy of the forward problem depends on the set of samples that is used in the training.
The samples are drawn randomly and we probably do not use sufficiently many samples to satisfy the restricted isometry properties required for the empirical reconstruction of the solution with high probability (cf.~\cite{EigelTrunschke2020}).
Consequently, the VMC algorithm might not find a strictly better (or equally good) approximation in each step.
In particular, we observe that error and estimator might increase after refinement of the spatial mesh $\mathcal{T}$ or enlarging the stochastic approximation space $\Lambda_d$ when the amount of training samples is not sufficiently increased yet.

We also observe that it is possible for the error and estimator to stagnate for several iterations before continuing convergence.
This is an artifact of the refinement strategy.
We increase the dimensions of modes by at most one in each step of the adaptive algorithm, which might not be sufficient for a reduction of the error.

\subsection{Representation complexity of the solution}

In this section we discuss the memory complexity of the TT representation of the solution for the experiments during the adaptive refinement of Algorithm~\ref{alg:aVMC}.
Figure~\ref{fig:ranks_dims_order_lshape} shows the growth in the number of active modes, their respective dimensions and the representation ranks of $u_{N,r,n}$ for the affine and lognormal case, respectively.
\begin{figure}[htpb]
  \begin{center}
    \includegraphics[width=\textwidth]{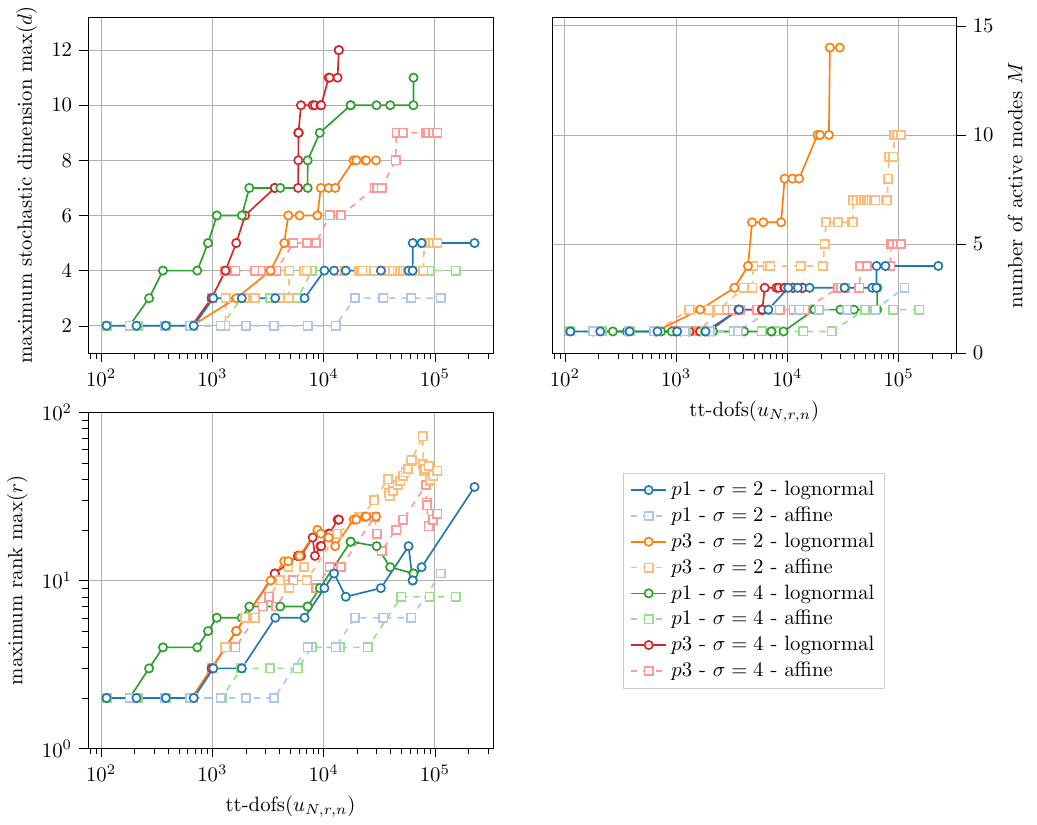}
  \end{center}
  \caption{%
    Maximum dimensions, ranks and number of active modes of the solution $u_{N,r,n}$ with respect to the $\operatorname{tt-dofs}$ of $u_{N,r,n}$ for the affine and lognormal case on the L-shaped domain.
  }%
  \label{fig:ranks_dims_order_lshape}
\end{figure}

In the affine case we observe the expected behaviour, namely that for the fast decay $\sigma=4$ the refinement Algorithm~\ref{alg:aVMC} focuses on increasing the dimensions for active modes whereas in the experiments with slow decay $\sigma=2$ activating new modes is prioritized independently on the computational domain.
We note that the dimension $d_1$ of the first mode is always the largest.
Moreover, the dimensions $d_m$ decrease monotonously as $m$ increases.
Algorithm~\ref{alg:aVMC} requires two to three times more iterations if cubic basis functions $p=3$ are used in the FE discretization.
This is again expected bahaviour since the mesh $\mathcal{T}$ needs to be refined less often due to the better approximation properties of higher order FE methods.
In the bottom left graph of Figure~\ref{fig:ranks_dims_order_lshape} we see that the maximal representation rank of the solution $u_{N,r,n}$ is up to six times larger for the experiments where we let $p=3$.
It seems that the maximal rank increases with an almost constant rate with respect to the $\operatorname{tt-dofs}(u_{N,r,n})$.
One of the first three ranks is always the largest with $r_1$ dominating in almost all cases.
The ranks $r_m$ decrease monotonously as $m$ increases beyond the index of the maximal rank.

In the lognormal case, Algorithm~\ref{alg:aVMC} behaves as expected as well, i.e.\ $\sigma=2$ requires the activation of more modes and $\sigma=4$ prioritizes larger polynomial degrees in the first modes.
The overall number of active modes is similar to the affine case and the maximal polynomial degree is slightly larger.
We suppose that the latter stems from the more complicated structure of $u_{N,r,n}$ for the lognormal diffusion coefficient $\kappa$.
The ranks also increase at an overall constant rate with respect to $\operatorname{tt-dofs}(u_{N,r,n})$ although the variation of coefficient realizations in case~\ref{case:setup:lognormal} is larger than in case~\ref{case:setup:affine}.
A comparison to~\cite{EMPS20} shows that the number of active modes and the maximal dimensions and ranks behave similarly in growth rate and magnitude.

Table~\ref{tab:number_of_samples} depicts the number of training samples $N_\mathrm{VMC}$ used in the last iteration before Algorithm~\ref{alg:aVMC} terminates.
Most notably we see that the lower order FE discretizations with $p=1$ require up to two orders of magnitude fewer training samples during the algorithm.
Moreover, the number of samples for the affine case with $p=1$ are not increased during any iteration of Algorithm~\ref{alg:aVMC}.
This can be explained by the high regularity of the solution $u$ of~\eqref{eq:setup:darcy} and the relatively large spatial error of the low-order FE discretization.
Opposite to this, Algorithm~\ref{alg:aVMC} terminates for the $p=3$ lognormal experiments as $N_{\mathrm{VMC}}$ exceeds $10^5$, which can be explained by the same argument.
This also causes the adaptive algorithm to terminate prematurely for fast decay $\sigma=4$ with $p=3$ in case~\ref{case:setup:lognormal}.

\begin{table}[htpb]%
    \centering
    \ra{1.1}  %
    \begin{tabular}{rcrrcrr}\toprule
             & & \multicolumn{2}{c}{\textbf{affine}} & \phantom{abc} & \multicolumn{2}{c}{\textbf{lognormal}}\\
      \cmidrule{3-4} \cmidrule{6-7}
             & & $\sigma=2$ & $\sigma=4$ & & $\sigma=2$ & $\sigma=4$\\
      \midrule
      $p=1$  & & $100$   & $100$  & & $1600$ & $5377$ \\
      $p=3$  & & $13280$ & $3280$ & & $3752$ & $162424$ \\
      \bottomrule
    \end{tabular}
    \caption{%
      Number of training samples $N_\mathrm{VMC}$ used in the last iteration of Algorithm~\ref{alg:aVMC} for all experiments on the L-shaped domain.
    }%
    \label{tab:number_of_samples}
\end{table}

\subsection{Refinement strategy evaluation}

In this section we take a detailed look at the concrete refinement decisions of Algorithm~\ref{alg:mark_and_refine}.
Figure~\ref{fig:level_refinement} depicts the error $\mathcal{E}_u(u_{N,r,n})$, the total estimator $\eta(u_{N,r,n})$ and the three estimator contributions $\eta_{\mathrm{det}}(u_{N,r,n})$, $\eta_{\mathrm{sto}}(u_{N,r,n})$ and $\eta_{\mathrm{alg}}(u_{N,r,n})$ exemplarily for the lognormal experiments on the L-shaped domain for $p=1$ with fast decay $\sigma=4$ (left) and $p=3$ with slow decay $\sigma=2$ (right).
The background patterns indicate by color which of the estimator parts dominate in the respective iteration.
The background hatches display the quantity that Algorithm~\ref{alg:mark_and_refine} chooses to increase based on the dominating estimator contribution and the estimated minimal Gramian eigenvalue.
To provide a more detailed view on the stochastic refinement, we distinguish between an increase of the dimension of already active modes and the activation of additional modes.

\begin{figure}[htp]
  \begin{center}
    \includegraphics[width=\textwidth]{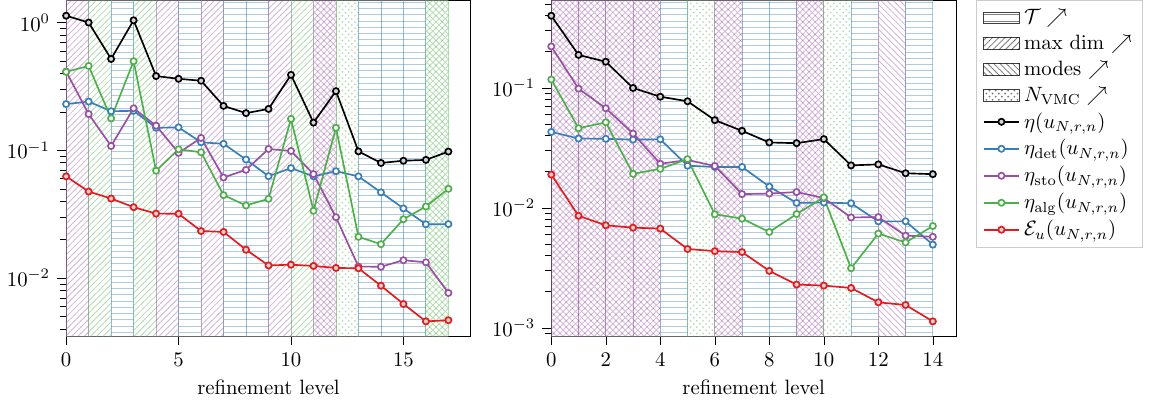}
  \end{center}
  \caption{%
    Error $\mathcal{E}_u(u_{N,r,n})$, total estimator $\eta(u_{N,r,n})$ and estimator contributions for each refinement level for the lognormal case on the L-shaped domain for $p=1$ with $\sigma=4$ (left) and $p=3$ with $\sigma=2$ (right).
    The background displays the dominating estimator contribution (color) and the resulting refinement strategy (hatching).
    The patterns are superimposed when multiple quantities are refined simultaneously.
  }%
  \label{fig:level_refinement}
\end{figure}

In both cases displayed in Figure~\ref{fig:level_refinement} we observe that each estimator contribution dominates in at least one iteration.
This confirms that adaptive refinement of finite element space, stochastic space and number of training samples is essential to reduce the approximation error.
In the case of the linear FE space $p=1$ with fast decay $\sigma=4$ (Figure~\ref{fig:level_refinement}, left), we note that even though $\eta_{\mathrm{alg}}(u_{N,r,n})$ dominates several times in the first few iterations, Algorithm~\ref{alg:resample} causes an increase of the VMC training samples only after extensively refining the spatial mesh $\mathcal{T}$ and the dimension of the first mode $d_1$.
Due to the fast decay $\sigma=4$, the number of modes has to be increased only once and only after extensive refinement of the other quantities.
In the case of the cubic FE space $p=3$ with slow decay $\sigma=2$ (Figure~\ref{fig:level_refinement}, right), the emphasis of Algorithm~\ref{alg:mark_and_refine} lies on the refinement of the stochastic space in the first iterations.

\section*{Acknowledgements}
M.\ Eigel acknowledges the partial support of the DFG SPP 1886 ``Polymorphic Uncertainty Modelling for the Numerical Design of Structures''.
N.\ Farchmin has received funding from the German Central Innovation Program (ZIM) No. ZF4014017RR7.
P.\ Trunschke acknowledges support by the Berlin International Graduate School in Model and Simulation based Research (BIMoS).
The authors would like to thank Robert Gruhlke, Manuel Marschall and Reinhold Schneider for fruitful discussions.

\begin{appendices}
    \section{Orthogonal polynomial basis functions}%
\label{supplement:normalization_and_triple_products}

\paragraph{Legendre polynomials}

The standard Legendre polynomials $\tilde L_j$ of degree $j\in\mathbb{N}_0$ constitute an orthogonal basis of $L^2(\Gamma_\ell,\pi_\ell)$ for $\Gamma_\ell=[-1,1]$ and $\pi_\ell(y_\ell)\equiv1/2\,\mathrm{d}y_{\ell}$.
With $c_j = \sqrt{2j+1}$ the set of polynomials $\{ L_j = c_j\tilde L_j \}_{j=0}^\infty$ is orthogonal and normalized with respect to $L^2(\Gamma_\ell,\pi_\ell)$.
To derive an analytical expression of the triple product $\tau_{ijk}=\mathbb{E}_{\pi_{\ell}}[L_i L_j L_k]$ of the normalized Legendre polynomials for any $i,j,k\in\mathbb{N}_0$, we define
\begin{align*}
  A_{\mathrm{Leg}}(n) = (2n)!\, 2^{-n} (n!)^{-2}.
\end{align*}
The triple product $\tau_{ijk}$ is then given by
\begin{equation}
  \label{eq:supplement:triple_product_leg}
  \begin{aligned}
    \tau_{ijk} = 
    \begin{cases}
      \frac{c_i c_j c_k}{(2s+1)}\,\frac{A_{\mathrm{Leg}}(s-i)A_{\mathrm{Leg}}(s-j)A_{\mathrm{Leg}}(s-k)}{A_{\mathrm{Leg}}(s)} & \mbox{if } s\in\mathbb{N}_0\mbox{ and }s\geq\max\{ i,j,k \},\\
      0 & \mbox{else},
    \end{cases}
  \end{aligned}
\end{equation}
where $s=(i+j+k)/2$.

\paragraph{Hermite polynomials}

The standard probabilists Hermite polynomials $\tilde H_j$ of degree $j\in\mathbb{N}_0$ constitute an orthogonal basis of $L^2(\Gamma_\ell,\pi_\ell)$ for $\Gamma_\ell=\mathbb{R}$ and $\pi_\ell(y_\ell)=(2\pi)^{-1/2}\exp(-y_\ell^2/2)\,\mathrm{d}y_{\ell}$.
With $c_j = 1 / \sqrt{j!}$ the set of polynomials $\{ H_j = c_j\tilde H_j \}_{j=0}^\infty$ is orthogonal and normalized with respect to $L^2(\Gamma_\ell,\pi_\ell)$.
To derive an analytical expression of the triple product $\tau_{ijk}=\mathbb{E}_{\pi_{\ell}}[H_i H_j H_k]$ of the normalized Hermite polynomials for any $i,j,k\in\mathbb{N}_0$, we define
\begin{align*}
  A_{\mathrm{Her}}(i,j,k) = \Bigl(\frac{i+j+k}{2}\Bigr)!.
\end{align*}
With the condition
\begin{align}
  \label{eq:supplement:condition_hermite}
  \frac{1}{2}(i+j+k)\in\mathbb{N}_0
  \quad\mbox{and}\quad
  \frac{1}{2}(i+j+k) \geq \max\{i,j,k\},
\end{align}
the triple product $\tau_{ijk}$ is given by
\begin{equation}
  \label{eq:supplement:triple_product_herm}
  \begin{aligned}
    \tau_{ijk} = 
    \begin{cases}
      c_i c_j c_k\,\frac{i!j!k!}{A_{\mathrm{Her}}(i,j,-k)A_{\mathrm{Her}}(i,-j,k)A_{\mathrm{Her}}(-i,j,k)} & \mbox{if } i,j,k \mbox{ satisfy~\eqref{eq:supplement:condition_hermite}},\\
      0 & \mbox{else}.
    \end{cases}
  \end{aligned}
\end{equation}

\paragraph{Scaled Hermite polynomials}

For Gaussian distributions $\mathcal{N}(0,\sigma^2)$ with $\sigma\neq1$ it is easy to scale the standard probabilists Hermite polynomials to obtain an orthogonal basis.
Recall from Section~\ref{sec:setup} that the univariate density for $\mathcal{N}(0,\sigma_\ell(\rho)^2)$ is given by
\begin{align*}
  \pi_\ell(y_\ell;\sigma_\ell(\rho)) = \zeta_{\ell}(y_\ell,\sigma_\ell(\rho))\, \pi_\ell(y_\ell,1)
\end{align*}
for $\sigma_\ell(\rho)=\exp(\rho\Vert\gamma_\ell\Vert_{L^\infty(D)})$ and $\rho\geq0$.
Note that $\sigma(\rho) = (\sigma_1(\rho), \sigma_2(\rho),\dots)\in\exp(\ell^1(\mathbb{N}))$ and define the transformation
\begin{align*}
  \iota_{\sigma(\rho)}\colon \mathbb{R}^\infty \to \mathbb{R}^\infty,
  \qquad (y_\ell)_{\ell\in\mathbb{N}} \mapsto (\sigma_\ell(\rho)^{-1} y_\ell)_{\ell\in\mathbb{N}}.
\end{align*}
With this we can define the multivariate scaled Hermite polynomials $H_\mu^{\rho} := H_\mu\circ\iota_{\sigma(\rho)}$, where $H_\mu$ is the multivariate and normalized standard probabilists Hermite polynomial.
The set $\{ H_\mu^{\rho} \}_{\mu\in\mathcal{F}}$ thus forms an orthogonal and normalized basis of $L^2(\Gamma_\kappa,\pi_\rho)$.
Moreover, for any $w\in L^2(\Gamma_\kappa,\pi_\rho)$ it holds 
\begin{align*}
  \int_{\Gamma_\kappa} w \,\mathrm{d}\pi_{0}(y)
  = \int_{\Gamma_\kappa} w\circ\iota_{\sigma(\rho)} \,\mathrm{d}\pi_{\rho},
\end{align*}
which implies that $\mathbb{E}_{\pi_{\rho}}[H_i^{\rho} H_j^{\rho} H_k^{\rho}] = \mathbb{E}_{\pi_0}[H_i H_j H_k]$ for any $i,j,k\in\mathbb{N}_0$.
Hence, the triple product with respect to the scaled Hermite polynomials can be computed by~\eqref{eq:supplement:triple_product_herm} as well.

\section{Well-posedness of the lognormal case}%
\label{supplement:well_posedness_lognormal}

First we note that by~\cite[Lemma 2.1]{Hoang2014} the set $\Gamma_{\kappa}$ from~\eqref{eq:setup:Gamma_kappa} is measurable and it holds $\pi_0(\Gamma_{\kappa})=1$.
Moreover, Lemma $2.2$ of~\cite{Hoang2014} shows that the lognormal field $\kappa$ is bounded and positive.
However, boundedness and positivity only hold pointwise and not uniformly over the parameter space $\Gamma_\kappa$.
Even though this leads to a far more intricate analysis, following the arguments of~\cite{SchwabGittelson2011} it is still possible to obtain a well-defined variational formulation by the introduction of the stronger measure $\pi_{\vartheta\rho}$ from~\eqref{eq:setup:zeta_and_pi} for some $\rho>0$ and $0<\vartheta<1$.
Recall that the bilinear form in the lognormal case~\ref{case:setup:lognormal} is given by
\begin{align*}
  B_{\vartheta\rho}(w,v) = \int_{\Gamma_\kappa} \int_{D}  \kappa \grad w\cdot\grad v \,\mathrm{d} x\,\mathrm{d} \pi_{\vartheta\rho}(y)
\end{align*}
and that the solution space is defined via
\begin{align*}
  \mathcal{V}_{\vartheta\rho} = \{ w\colon\Gamma_\kappa\to\mathcal{X} \mbox{ measurable with }B_{\vartheta\rho}(w,w)<\infty \}.
\end{align*}
By~\cite[Proposition 2.43]{SchwabGittelson2011} it then follows that
\begin{align*}
  L^2(\Gamma_\kappa,\pi_{\rho};\mathcal{X})
  \subset \mathcal{V}_{\vartheta\rho}
  \subset L^2(\Gamma_\kappa,\pi_{0};\mathcal{X})
  \qquad\mbox{for any }0<\vartheta<1
\end{align*}
are continuous embeddings.
With this, Lemma~$2.41$ and Lemma~$2.42$ from~\cite{SchwabGittelson2011} show that the bilinear form $B_{\vartheta\rho}$ is $\mathcal{V}_{\vartheta\rho}$--elliptic and bounded in the sense that
\begin{align}
  \label{eq:supplement:boundedness_B}
  \vert B_{\vartheta\rho}(w,v) \vert
  &\leq \hat{c}(\vartheta\rho) \Vert w\Vert_{L^2(\Gamma_\kappa,\pi_{\rho};\mathcal{X})} \Vert v\Vert_{L^2(\Gamma_\kappa,\pi_{\rho};\mathcal{X})}
  &\mbox{for all }w,v\in L^2(\Gamma_\kappa,\pi_{\rho};\mathcal{X}),\\
  \label{eq:supplement:coercivity_B}
  B_{\vartheta\rho}(w,w) 
  &\geq \check{c}(\vartheta\rho) \Vert w\Vert_{L^2(\Gamma_\kappa,\pi_{0};\mathcal{X})}^2
  &\mbox{for all }w\in L^2(\Gamma_\kappa,\pi_{0};\mathcal{X}).
\end{align}

\section{The Tensor Train format}%
\label{supplement:tt_format}

In the following we briefly recall the TT format and some fundamental properties of TTs.
For a more detailed overview, we refer the reader to~\cite{Bachmayr2016,Oseledets2009,Holtz2012,EigelPfeffer2016,Rohrbach2020} and the references therein.
Any function $w_N\in\mathcal{V}_N(\Lambda_d;\mathcal{T},p)$ has an expansion of the form
\begin{align}
  \label{eq:discretization:expansion_of_function}
  w_N(x,y) = \sum_{j\in[\vert\mathcal{T}\vert]} \sum_{\mu\in\Lambda_d} \discretized{w}[j,\mu] \varphi_j(x) P_\mu(y)
  \qquad\mbox{with}\qquad
  \discretized{w} \in \mathbb{R}^{\vert\mathcal{T}\vert\times d_1 \times\dots\times d_M}.
\end{align}
Hence, $\mathcal{V}_N$ is isomorphic to the space of coefficient tensors $\mathbb{R}^{\vert\mathcal{T}\vert\times d} = \mathbb{R}^{\vert\mathcal{T}\vert\times d_1\times\dots\times d_M}$.
Note that the size of the coefficient tensor $\discretized{w}$ grows exponentially with the order $M$, which is commonly referred to as the \emph{curse of dimensionality}.
To mitigate this exponential dependence on $M$, we employ a low-rank decomposition of the tensor $\discretized{w}$.
There are many tensor decompositions available~\cite{Grasedyck2013,Hac12,Khoromskij2015,Kolda2009}, but due to its simplicity and the wide availability in numerical libraries we chose the TT format for our derivations.
The TT representation of a tensor $\discretized{w}\in\mathbb{R}^{J\times d}$, $J\in\mathbb{N}$, is given by
\begin{align}
  \label{eq:discretization:TTformat_of_function}
  \discretized{w}[j,\mu]
  &= \sum_{k_1=1}^{r_1}\dots\sum_{k_M=1}^{r_M}\discretized{w}_0[j,k_1]\prod_{m=1}^{M} \discretized{w}_m[k_m,\mu_m,k_{m+1}]
  \quad\mbox{for any }j\in[J]\mbox{ and }\mu\in \Lambda_d\\
  &=: \sum_{k=1}^{r}\discretized{w}_0[j,k_1]\prod_{m=1}^{M} \discretized{w}_m[k_m,\mu_m,k_{m+1}]
\end{align}
for some $r\in\mathbb{N}^M$, where we use the convention $r_{M+1}=1$.
In our application the zeroth component tensor $\discretized{w}_0\in\mathbb{R}^{J\times r_1}$ corresponds to the spatial discretization.
The stochastic contributions are given as order three tensors $\discretized{w}_m\in\mathbb{R}^{r_m\times d_m\times r_{m+1}}$ for each mode.
If all ranks $r_m$ are minimal, this is called \emph{tensor train decomposition} of $\discretized{w}$ with TT rank $r$ and we write $\operatorname{tt-rank}(\discretized{w}) = r$.
The set of all tensors of TT rank $r$ forms a manifold~\cite{HRS12} of dimension
\begin{align*}
  \operatorname{tt-dofs}(\discretized{w})
  = J r_{1} - r_{1}^{2} + \sum_{m=1}^{M-1} \bigl(r_m d_m r_{m+1} - r_{m+1}^{2}\bigr) + r_M d,
\end{align*}
which shows that the complexity of the TT format behaves like $\mathcal{O}(J\hat{r}+M\hat{d}\hat{r}^2)$ for $\hat{d}=\max\{ d_1,\ldots,d_M \}$ and $\hat{r}=\max\{ r_1,\ldots,r_M \}$.
In contrast to full tensor representations, with complexity $\mathcal{O}(J\hat{d}^{M})$, TTs depend only linearly on the order $M$.
As a result, the TT format is especially efficient for a small maximal rank $\hat{r}$.

Similarily we can express linear operators $W\colon \mathcal{V}_N(\Lambda_d;\mathcal{T},p)\to\mathcal{V}_{N}(\Lambda_q;\mathcal{T},p)$ in the TT format.
For this recall that the application of $W$ to $v_N\in\mathcal{V}_N(\Lambda_d;\mathcal{T},p)$ yields
\begin{align*}
  Wv_N(x,y) = \sum_{i\in[N]} \sum_{\nu\in\Lambda_q} \sum_{j\in[N]} \sum_{\mu\in\Lambda_d} \discretized{W}[i,\nu;j,\mu]\discretized{v}[j,\mu] \phi_i(x)P_{\nu}(y).
\end{align*}
The TT representation of the tensor operator $\discretized{W}\colon\mathbb{R}^{J\times d}\to\mathbb{R}^{J\times q}$ is thus determined by
\begin{align*}
  \discretized{W}[i,\nu;j,\mu]
  &= \sum_{k=1}^{r} \discretized{W}_0[i,j,k_{1}] \prod_{m=1}^{M} \discretized{W}_m[k_m,\nu_m,\mu_m,k_{m+1}]
\end{align*}
for any $i,j\in[J]$, $\mu\in\Lambda_d$ and $\nu\in\Lambda_q$ with component tensors $\discretized{W}_0\in\mathbb{R}^{J\times J\times r_{1}}$ and $\discretized{W}_m\in\mathbb{R}^{r_m\times q_m\times d_m\times r_{m+1}}$.
The TT decomposition always exists and can be computed using the hierarchical singular value decomposition (SVD)~\cite{HRS12}.
A truncated hierarchical SVD leads to quasi-optimal approximations of the TT decomposition in the Frobenius norm~\cite{OT09,Gra09,Loe78,HS14}.
This can also be applied to tensors which are already given in the TT format to obtain a TT decomposition with a lower rank.
This process is referred to as \emph{rounding}.

\section{Experiments on the unit square domain}%
\label{supplement:unit_square}

Additional figures for convergence of the adaptive algorithm and complexity of the solution $u_{N,r,n}$ on the unit square $D=(0,1)^2$.
Details on the experiments are described in Section~\ref{sec:experiments}.

\begin{figure}[htp]
  \begin{center}
    \includegraphics[height=4.5cm]{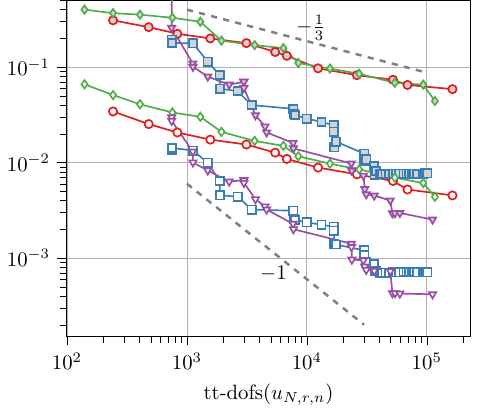}
    \includegraphics[height=4.5cm]{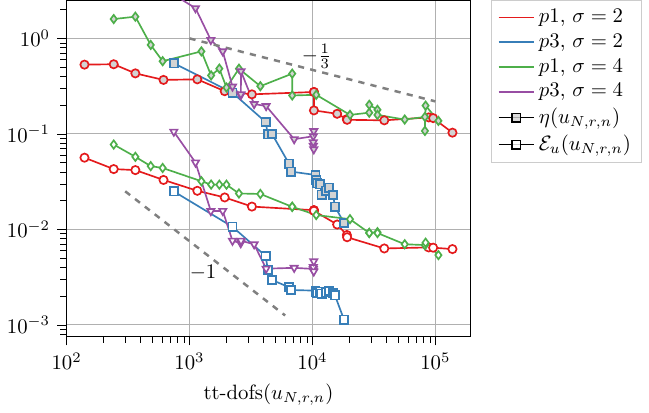}
  \end{center}
  \caption{%
    Sampled root mean square error $\mathcal{E}_u(u_{N,r,n})$ and total error estimator $\eta(u_{N,r,n})$ of the fully adaptive algorithm on the unit square domain.
    Considered are finite element approximations of order $p=1$ and $p=3$ for slow ($\sigma=2$) and fast ($\sigma=4$) decay for the affine (left) and lognormal (right) case.
  }%
  \label{fig:convergence_unit_square}
\end{figure}

\begin{figure}[htpb]
  \begin{center}
    \includegraphics[height=9cm]{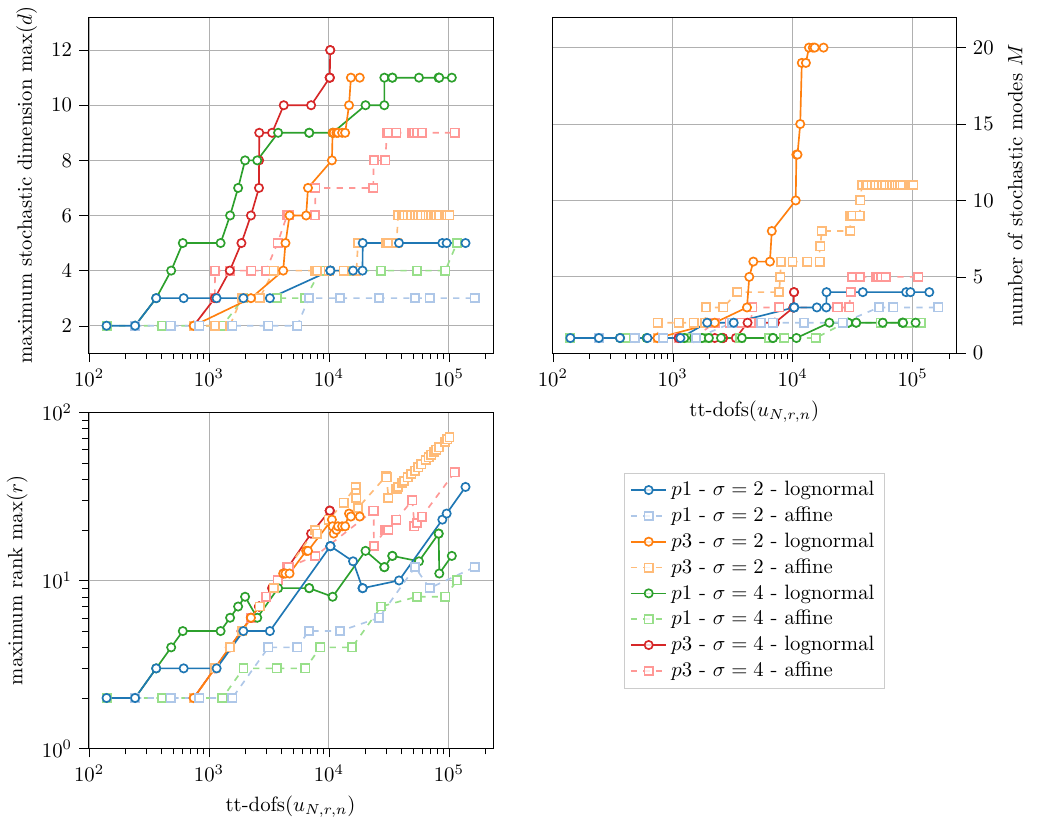}
  \end{center}
  \caption{%
    Maximum dimensions, ranks and number of active modes of the solution $u_{N,r,n}$ with respect to the $\operatorname{tt-dofs}$ of $u_{N,r,n}$ for the affine and lognormal case on the unit square domain.
  }%
  \label{fig:ranks_dims_order_unitSquare}
\end{figure}

\end{appendices}

\printbibliography

\end{document}